\documentclass[12pt,oneside,a4paper]{amsart}
\usepackage{microtype}
\usepackage{amssymb}
\usepackage{hyperref}
\usepackage{tikz}

\usepackage[T1]{fontenc}
\usepackage{newtxtext,newtxmath}

\renewcommand{\epsilon}{\varepsilon}
\renewcommand{\phi}{\varphi}

\newtheorem{corollary}{Corollary}
\newtheorem{lemma}{Lemma}
\newtheorem{theorem}{Theorem}
\newtheorem{definition}{Definition}
\newtheorem{remark}{Remark}

\makeatletter
\renewcommand\subsection{\@startsection{subsection}{2}%
  \z@{.5\linespacing\@plus.7\linespacing}{.1\linespacing}%
  {\normalfont\scshape}}
\makeatother

\title[Strict convexity for GJE]{Strict $g$-convexity for generated Jacobian equations with applications to global regularity}
\author{Cale Rankin}
\email{cale.rankin@anu.edu.au\\cale.rankin@gmail.com}
\thanks{This research is supported by an Australian Government Research Training Program (RTP) Scholarship and by ARC DP 200101084.}
\begin{document}

\begin{abstract}
This article has two purposes. The first is to prove solutions of the second boundary value problem for generated Jacobian equations (GJEs) are strictly $g$-convex. The second is to prove the global $C^3$ regularity of Aleksandrov solutions to the same problem. In particular, Aleksandrov solutions are classical solutions. These are related because the strict $g$-convexity is essential for the proof of the global regularity. The assumptions for the strict $g$-convexity are the natural extension of those used by Chen and Wang in the optimal transport case. They are the Loeper maximum principle condition, a positively pinched right-hand side, a $g^*$-convex target, and a source domain strictly contained in a $g$-convex domain. This improves the existing domain conditions though at the expense of requiring a $C^3$ generating function. This is appropriate for global regularity where existence is proved assuming a $C^4$ generating function.  We prove the global regularity under the hypothesis that Jiang and Trudinger recently used to obtain the \textit{existence} of a globally smooth solution and an additional condition on the height of solutions. Our proof of global regularity is by modifying Jiang and Trudinger's existence result to construct a globally $C^3$ solution intersecting the Aleksandrov solution. Then the strict convexity yields the interior regularity to apply the author's uniqueness results.
\end{abstract}

 \maketitle
 \section{Introduction}

Generated Jacobian equations generalise the Monge--Amp\`ere equation to encompass applications to geometric optics. For these equations a generalised notion of convexity, known as $g$-convexity, plays the same role as convexity does for Monge--Amp\`ere equations. In this article we prove the strict $g$-convexity of solutions to the second boundary value problem for GJEs. As an application of this result we prove the global $C^3 $ regularity of Aleksandrov solutions to the same problem under stronger hypothesis. 

Generated Jacobian equations are PDE of the form
\begin{equation}
  \label{eq:gje}
   \det DY(\cdot,u,Du) = \psi(\cdot,u,Du) \text{ in }\Omega,
 \end{equation}
 where $Y:\mathbf{R}^n\times \mathbf{R}\times\mathbf{R}^n \rightarrow \mathbf{R}^n$ is a vector field with a particular structure: It arises from a generating function as outlined in Section \ref{sec:g}. The other components of \eqref{eq:gje} are a nonnegative function $\psi: \mathbf{R}^n\times \mathbf{R}\times\mathbf{R}^n \rightarrow \mathbf{R}$ and a domain $\Omega \subset\mathbf{R}^n$. For the strict $g$-convexity we assume
\begin{equation}
  \label{eq:pinch-rhs}
   \lambda \leq \psi(\cdot,u,Du) \leq \Lambda,
\end{equation}
for positive constants $\lambda,\Lambda$. For the global regularity we assume
\begin{equation}
  \label{eq:frac-rhs}
   \psi(\cdot,u,Du) = \frac{f(\cdot)}{f^*(Y(\cdot,u,Du))},
\end{equation}
for $C^2$ positive functions $f,f^*$. In each case \eqref{eq:gje} is coupled with the second boundary value problem:
\begin{equation}
  \label{eq:2bvp}
   Y(\cdot,u,Du)(\Omega) = \Omega^*,
\end{equation}
for a prescribed domain $\Omega^*.$ We require convexity conditions on $\Omega,\Omega^*$. These, and other required definitions and structure conditions, are introduced in Section \ref{sec:g} where we also state our main results: Theorems \ref{thm:strict-convexity} and \ref{thm:global-regularity}.

Generated Jacobian equations were introduced by Trudinger \cite{Trudinger14} to extend the theory of Monge--Amp\`ere type equations in optimal transport to problems in geometric optics. Thus, to situate our results we outline the optimal transport case.

The Monge--Amp\`ere equation, recovered from \eqref{eq:gje} by taking $Y(x,u,Du) = Du$, is a fully nonlinear PDE, elliptic when $u$ is convex. Brenier \cite{Brenier1991} showed that  a (suitably defined) weak solution of \eqref{eq:gje} subject to \eqref{eq:2bvp} is obtained by solving the optimal transport problem with quadratic cost. Thus studying the regularity of optimal transport maps is tantamount to studying the Monge--Amp\`ere equation paired with the second boundary value problem. The work relevant to ours is that of Caffarelli \cite{Caffarelli92a,Caffarelli} and Urbas \cite{Urbas1997}. Caffarelli's work takes place under the assumption \eqref{eq:pinch-rhs} and is concerned with the strict convexity and $C^{1,\alpha}$ regularity of a weak notion of solution, known as Aleksandrov solutions. Urbas, in the smooth setting, obtained the global estimates required for the method of continuity, thereby proving the existence of globally regular solutions. The uniqueness of solutions (up to a constant) then implies the global regularity of Aleksandrov solutions. 

The optimal transport problem with a general cost yields an equation of the form \eqref{eq:gje} with $Y(\cdot,u,Du) = Y(\cdot,Du)$. This  equation is more complex.  However thanks to new ideas  built on a generalised notion of convexity and a condition known as A3w, the results of Caffarelli and of Urbas have been extended to this new equation. Urbas's results were extended by Trudinger and Wang \cite{TrudingerWang09}. They obtained the estimates for the method of continuity and proved the existence of globally $C^3$ solutions. Again, by the uniqueness up to a constant, this implies global regularity of Aleksandrov solutions. The Caffarelli style strict convexity and $C^{1,\alpha}$ regularity under A3w has been proved by a number of authors, namely Figalli, Kim, and McCann \cite{FKM13}, Guillen and Kitagawa \cite{GuillenKitagawa15}, V\`etois\cite{Vetois15}, and Chen and Wang \cite{ChenWang16}. Each proves the strict convexity and $C^{1,\alpha}$ regularity under different hypotheses on the domain and cost function. 

For GJEs the picture is not complete. Guillen and Kitagawa \cite{GuillenKitagawa17} have extended their strict convexity and $C^{1,\alpha}$ regularity results to GJEs. Their result is of particular interest because their regularity requirements on the generating function are very weak --- it need only be $C^2$. However there are situations where their domain conditions are restrictive. One of which is our application to global regularity. Thus it's of interest to have the extension of Chen and Wang's result, which holds under weaker conditions on the domains, to generated Jacobian equations. This extension is one of the goals of this paper. The strict $g$-convexity implies $C^{1,\alpha}$ regularity, as shown by Guillen and Kitagawa. We apply the strict convexity result to the problem of  $C^3(\overline{\Omega})$ regularity for Aleksandrov solutions. The existence of $C^3(\overline{\Omega})$ solutions was proved by Jiang and Trudinger \cite{JiangTrudinger18}. However, Karakhanyan and Wang \cite{KarakhanyanWang2010} have shown solutions of generated Jacobian equations can have different regularity properties. Thus, in stark contrast to the Monge--Amp\`ere  and optimal transport setting, there is no uniqueness up to a constant and the existence of a globally smooth solution does not imply the regularity of Aleksandrov solutions. Recently the author proved if two $C^{1,1}_{\text{loc}}$ solutions intersect then they are the same solution \cite{Rankin2020}. We combine this with the strict convexity and Jiang and Trudinger's existence result to prove the global regularity of Aleksandrov solutions as follows: First we introduce a modification of Jiang and Trudinger's construction so as to obtain a $C^3(\overline{\Omega})$ solution intersecting a given Aleksandrov solution. By the strict $g$-convexity and Trudinger's recent interior regularity result for strictly $g$-convex solutions \cite{Trudinger20} we have the required regularity to apply the uniqueness result. Thus our Aleksandrov solution is the constructed $C^3(\overline{\Omega})$ solution.

Here's the outline of the paper. In Section \ref{sec:g} we introduce $g$-convexity and the required definitions for our main results. In Section \ref{sec:transformations-1} we introduce a transformation to the generating function and coordinates which makes the generating function ``almost affine''. Such transformations are used in the optimal transport case for strict convexity \cite{ChenWang16} and interior regularity \cite{LTW10,LTW15}.  These transformations are the main tool required to extend Chen and Wang's strict $c$-convexity; transformations in hand our strict $g$-convexity result follows a similar framework and uses similar proofs to theirs. We introduce $g$-cones in Section \ref{sec:g-cones-subd} and estimate their $g$-subdifferentials. In Section \ref{sec:uniform-estimates} we obtain uniform estimates for Aleksandrov solutions of the Dirichlet problem for GJEs. The strict $g$-convexity is proved in Section \ref{sec:strict-conv} followed by a short proof of $C^1$ differentiability in Section \ref{sec:furth-cons-strict}. We note at two critical junctures we rely on Lemmas from \cite{FKM13} and \cite{GuillenKitagawa17} and our $g$-cone and uniform estimates have been obtained under different hypotheses in \cite{GuillenKitagawa17}.  Finally in Section \ref{sec:global-regularity} we complete the proof of  $C^3(\overline{\Omega})$ regularity of Aleksandrov solutions as outlined above.

\subsubsection*{Acknowledgements}
My thanks to Shibing Chen and Xu-Jia Wang for discussions regarding \cite{ChenWang16}. 

\section{Generating functions and $g$-convexity}
\label{sec:g}

The theory of generated Jacobian equations is a combination of elliptic PDE and a generalisation of  convexity theory. Here we give the definitions required for the generalised convexity theory. More detailed introductions can be found in \cite{GuillenKitagawa17,Guillen19,RankinThesis}. We begin with the definition of generating functions. These are a nonlinear extension of affine supporting planes. 

\begin{definition}
  A generating function is a function $g$ satisfying the conditions A0,A1,A1$^*$, and A2. 
\end{definition}

\textbf{A0.} The function $g$ satisfies $g \in C^3(\overline{\Gamma})$ where $\Gamma$ is a bounded domain of the form $\Gamma := \{(x,y,z); x \in U, y \in V, z \in I_{x,y}\} \subset \mathbf{R}^n \times \mathbf{R}^n \times \mathbf{R}$ for domains $U,V \subset \mathbf{R}^n$ and  $I_{x,y}$ an open interval for each $x,y$. Moreover we assume there is an open interval $J$ such that $g(x,y,I_{x,y}) \supset J$ for each $x \in U, y \in V$. 

\textbf{A1. } For each $(x,u,p) \in \mathcal{U}$ defined by
\[ \mathcal{U} = \{(x,g(x,y,z),g_x(x,y,z)) ; (x,y,z) \in \Gamma\},\]
there is a unique $(x,y,z) \in \Gamma$ such
\begin{align*}
  &g(x,y,z) = u &&g_x(x,y,z)=p.
\end{align*}

\textbf{A1$^*$. } For each fixed $y,z$ the mapping $x \mapsto \frac{g_y}{g_z}(x,y,z)$ is injective on its domain of definition.

\textbf{A2.} On $\overline{\Gamma}$ there holds $g_z<0$ and  the matrix\footnote{Subscripts before the comma denote differentiation with respect to $x$, subscripts after the comma (which are not $z$) denote differentiation with respect to $y$.}
\[ E:= g_{i,j}-g_z^{-1}g_{i,z}g_{,j}\]
satisfies $\det E \neq 0$. 

Later we introduce a dual generating function and we'll see the A1$^*$ condition is simply the A1 condition for the dual generation function --- thereby justifying the name. We've incorporated part of Guillen and Kitagwa's definition of uniform admissibility condition into A0. The boundedness requirement on $\Gamma$ may be weakened provided $\Vert g\Vert_{C^3(\Gamma)}$ is finite and the quantities in A2 are bounded away from zero. 

We define mappings $Y:\mathcal{U} \rightarrow \mathbf{R^n}, Z:\mathcal{U} \rightarrow \mathbf{R}$ by requiring they solve
\begin{align}
  \label{eq:yz-def1} g(x,Y(x,u,p),Z(x,u,p))&=u,\\
  \label{eq:yz-def2} g_x(x,Y(x,u,p),Z(x,u,p))&=p. 
\end{align}

We domains $\Omega\subset U$ and $\Omega^* \subset V$ and do not  forbid $\Omega = U, \Omega^* = V$.

\begin{definition}
  A generated Jacobian equation is an equation of the form \eqref{eq:gje} where the vector field $Y$ arises from solving \eqref{eq:yz-def1},\eqref{eq:yz-def2} for a generating function. 
\end{definition}

The basic example is the Monge--Amp\`ere equation which arises from $g(x,y,z) = x\cdot y - z$. The Monge--Amp\`ere equation is elliptic when solutions are convex. In this case, that's when solutions are supported by the generating function at each point. This permits the following generalisation. 

\begin{definition}
  A function $u:\Omega \rightarrow \mathbf{R}$ is called $g$-convex (strictly $g$-convex) provided for every $x_0 \in \Omega$ there is $y_0 \in V$ and $z_0 \in \cap_{x \in \Omega}I_{x,y_0}$ such that
  \begin{align}
\label{eq:conv-def1}    u(x_0) &= g(x_0,y_0,z_0)\\
\label{eq:conv-def2}    u(x) &\geq(>) g(x,y_0,z_0) \text{ for all } x \in \Omega, x \neq x_0,
  \end{align}
  and for any such $x_0,y_0,z_0$ we have
  \begin{equation}
    \label{eq:containment}
    g(\overline{\Omega},y_0,z_0)\subset J.
  \end{equation}
\end{definition}
The containment condition is due to Guillen and Kitagawa \cite{GuillenKitagawa17} and $g$-convex functions satisfying it are referred to by them as ``very nice''. Our definition of $g$-convex functions implies they are semiconvex. 

For $g$-convex $u:\Omega\rightarrow \mathbf{R}$ we define a mapping $Yu:\Omega \rightarrow V$ as follows
\[ Yu(x_0) = \{y \in V; \text{there is } z_0 \text{ such that \eqref{eq:conv-def1}, \eqref{eq:conv-def2} and \eqref{eq:containment} hold}\}.\]
If $u$ is differentiable $Yu(x) = Y(x,u,Du)$. Thus $Yu$ generalises the mapping $x \mapsto Y(x,u,Du)$ in much the same way the subgradient generalises the gradient (see Lemma \ref{lem:loeper-consequences} for details).  We restrict our attention to $g$-convex solutions of GJEs. For such solutions the PDE is degenerate elliptic \cite{Trudinger14}. Using the $Y$ mapping we have the follow generalisation of Aleksandrov solutions.

\begin{definition}
  Let $u:\Omega \rightarrow \mathbf{R}$ be a $g$-convex function. Then $u$ is called an Aleksandrov solution of \eqref{eq:gje} provided for every Borel $E\subset \Omega$ there holds
  \[ \vert Yu(E) \vert = \int_{E} \psi(\cdot,u,Du),\]
  where $Du$ is well defined almost everywhere by the semiconvexity. When $\psi$ has the form (\ref{eq:frac-rhs}) for positive $f,f^*$, it is equivalent to require
  \[\int_{Yu(E)}f^*(y) \ dy = \int_E f(x) \ dx.\]
\end{definition}

Given $y_0 \in Yu(x_0)$ we frequently need to find the $z_0$ for which $g(\cdot,y_0,z_0)$  is a $g$-support at $x_0$. This is accomplished by the dual generating function, which plays several important roles.

\begin{definition}
  The dual generating function is the unique function $g^*$-defined on
  \[ \Gamma^*:= \{(x,y,u) = (x,y,g(x,y,z)); (x,y,z) \in \Gamma\},\]
  by either of the equivalent requirements
  \begin{equation}
    \label{eq:gstar}
     g(x,y,g^*(x,y,u)) = u \text{ or, equivalently, } g^*(x,y,g(x,y,z)) = z. 
  \end{equation}
\end{definition}
The dual generating function is well defined because $g_z<0$. We note  if $g(\cdot,y_0,z_0)$ is a support at $x_0$ by \eqref{eq:conv-def1} we have $z_0 = g^*(x_0,y_0,u(x_0))$ and subsequently the support is $g(\cdot,y_0,g^*(x_0,y_0,u(x_0)))$. Differentiating \eqref{eq:gstar} we obtain the identities
\begin{align}
  \label{eq:gstar-deriv}
  &g^*_u = \frac{1}{g_z}, &&g^*_x = - \frac{g_x}{g_z}, &&g^*_y = -\frac{g_y}{g_z}, 
\end{align}
where $g^*$-terms are evaluated at $(x,y,u)$ and $g$-terms at $(x,y,g^*(x,y,u))$, or, alternatively, $g$-terms are evaluated at $(x,y,z)$ and $g^*$-terms at $(x,y,g(x,y,z))$. 

We also have domain convexity notions.

\begin{definition}
  \begin{enumerate}
  \item A set $A \subset U$ is called (uniformly) $g$-convex with respect to $y \in V,z \in \cap_{x \in A}I_{x,y}$ provided
    \[ \frac{g_y}{g_z}(A,y,z),\]
    is (uniformly) convex. \\
    \item A set $B \subset V$ is called (uniformly) $g^*$-convex with respect to $x \in U,u \in J$ provided
    \[g_x(x,\cdot,g^*(x,\cdot,u))(B),\]
    is (uniformly) convex.
  \end{enumerate}
\end{definition}

Certain statements are made more concise by defining $g$-convexity with respect to a function, as opposed to points.

\begin{definition} Let $u:\Omega\rightarrow \mathbf{R}$ be a $g$-convex function.
  \begin{enumerate}
  \item A set $A \subset U$ is called (uniformly) $g$-convex with respect to $u$ provided
    \[ \frac{g_y}{g_z}(A,y,z),\]
    is (uniformly) convex whenever $y \in Yu(x), z = g^*(x,y,u(x))$ for $x \in \Omega$ ($x \in \overline{\Omega}$). \\
    \item A set $B \subset V$ is called (uniformly) $g^*$-convex with respect to $u$ provided
    \[g_x(x,\cdot,g^*(x,\cdot,u(x)))(B),\]
    is (uniformly) convex for each $x \in \Omega$ ($x \in \overline{\Omega}$).
  \end{enumerate}
\end{definition}

In line with the above definitions of $g/g^*$-convexity, sets whose image under $x \mapsto \frac{g_y}{g_z}(x,y,z)$ is a line segment, will be used repeatedly.
\begin{definition}
  \begin{enumerate}
  \item   A collection of points $\{x_\theta\}_{\theta \in [0,1]} \subset U$ is called a $g$-segment with respect to $y,z$ provided
  \[\left\{\frac{g_y}{g_z}(x_\theta,y,z)\right\}_{\theta \in [0,1]}, \]
  is a line segment. \\
\item A collection of points $\{y_\theta\}_{\theta \in [0,1]} \subset V$ is called a $g^*$-segment with respect to $x,u$ provided
 $\{g_x(x,y_\theta,g^*(x,y_\theta,u))\}_{\theta \in [0,1]}$
  is a line segment. 
  \end{enumerate}
\end{definition}

These are the basic definition of $g$-convexity theory. In the optimal transport case Ma, Trudinger, and Wang \cite{MTW05} introduced a condition on the fourth derivatives of the cost function to prove interior regularity. Later work \cite{Loeper09,Liu09,TrudingerWang09a,KimMcCann10,FKM13} revealed that this condition is essential for the convexity theory (outlined in Lemma \ref{lem:loeper-consequences}). Loeper found a synthetic  interpretation of the condition which has the following extension to generating functions.\\

\textbf{The Loeper Maximum principle (LMP). }
Let $x_0 \in U$, $y_0,y_1 \in V$ and $u_0 \in J$ be given. Let $\{y_\theta\}_{\theta \in [0,1]}$ denote the $g^*$-segment with respect $x_0,u_0$ that joins $y_0$ to $y_1$. The generating function $g$ satisfies the Loeper maximum principle provided for all $x \in U$
\[ g(x,y_\theta,g^*(x_0,y_\theta,u_0)) \leq \text{max}\{g(x,y_1,g^*(x_0,y_1,u_0)),g(x,y_0,g^*(x_0,y_0,u_0))\}.\]

 The following results, concerning compatibility between the definitions of $g$-convex functions and $g$-convex sets, are well known consequences of the Loeper maximum principle \cite{LoeperTrudinger21,Trudinger20}.

\begin{lemma}\label{lem:loeper-consequences}
  Assume $g$ is a generating function satisfying the Loeper maximum principle and $u:\Omega \rightarrow \mathbf{R}$ is a $g$-convex function. Let $x_0 \in \Omega, y_0 \in Yu(x_0),u_0=u(x_0)$ and $z_0 = g(x_0,y_0,u_0)$. The following statements hold.
  \begin{enumerate}
  \item $Yu(x_0)$ is $g^*$-convex with respect to $x_0,u_0$. \\
  \item Let $h>0$. Then the sets $\{x \in \Omega; u(x) < g(x,y_0,z_0-h)\}$ and $\{x \in \Omega ;u(x) = g(x,y_0,z_0)\}$ are, when compactly contained in $\Omega$,  $g$-convex with respect to $y_0,z_0-h$, and $y_0,z_0$ respectively. \\
  \item $Yu(x_0) = Y(x_0,u_0,\partial u(x_0))$ where $\partial u$ denotes the subdifferential  \footnote{The subdifferential is defined for a semiconvex function $u:\Omega\rightarrow \mathbf{R}$ by
 \[ \partial u(x_0) = \{ p \in \mathbf{R}^n; u(x) \geq p \cdot (x-x_0)+u(x_0)+o(\vert x-x_0\vert ) \text{ for all }x \in \Omega\}.\]}. 
  \end{enumerate}
\end{lemma}

Now we have the required terminology to state the conditions for strict convexity. We assume $\lambda, \Lambda \in \mathbf{R}$ are positive. 

\begin{theorem}\label{thm:strict-convexity}
  Assume $g$ is a generating function satisfying the Loeper maximum principle. Assume $\overline{\Omega} \subset U$, and $u:\Omega \rightarrow \mathbf{R}$ is a $g$-convex function satisfying that for every $E \subset \Omega$
  \begin{equation}
    \label{eq:caff-style-ineq}
       \lambda \vert E\vert \leq \vert Yu(E)\vert \leq \Lambda \vert E \vert. 
  \end{equation}
  If $U$ and $Yu(\Omega)$ are, respectively,  $g$ and $g^*$-convex with respect to $u$, then $u$ is strictly $g$-convex. 
\end{theorem}
It is clear that by redefining $\Gamma$ it suffices there exist any $U' \subset U$ with $\Omega \subset\subset U'$ and $U'$ $g$-convex with respect to $u$.

We give a short proof of the following consequence of strict $g$-convexity.
\begin{corollary}\label{cor:c1}
  Assume $g$ is a generating function satisfying the Loeper maximum principle. Assume $u:\Omega \rightarrow \mathbf{R}$ is a strictly $g$-convex solution of \eqref{eq:caff-style-ineq}. Then $u \in C^{1}(\Omega)$. 
\end{corollary}

Guillen and Kitagawa have proved a stronger conclusion, that strictly $g$-convex solutions of \eqref{eq:caff-style-ineq} are in $C^{1,\alpha}(\Omega)$. An obvious consequence is the $C^{1,\alpha}(\Omega)$ regularity holds under the new domain hypotheses in Theorem \ref{thm:strict-convexity}.

Our global regularity result uses two additional conditions, called A4w and A5, on the generating function. We introduce these in Section \ref{sec:global-regularity}, though note the A5 condition bounds the gradient of solutions in terms of a particular constant $K_0$. For global regularity the PDE is 
\begin{equation}
  \label{eq:gjef}
  \det DY(\cdot,u,Du) = \frac{f(\cdot)}{f^*(Y(\cdot,u,Du))} \quad \text{in }\Omega,
\end{equation}
for positive densities $f,f^*$ and we assume the second boundary value problem \eqref{eq:2bvp} is satisfied. Then a necessary condition for the existence of a $C^2(\Omega)$ $g$-convex solution is the mass balance condition
\begin{equation}
  \label{eq:mb}
  \int_{\Omega}f = \int_{\Omega^*}f^*.
\end{equation}

\begin{theorem}\label{thm:global-regularity}
  Let $\overline{\Omega} \subset U$, $\overline{\Omega^*} \subset V$ be $C^4$ domains, $g$ a $C^4$ generating function satisfying LMP, A4w, A5, and, finally,  $f \in C^2(\overline{\Omega}),f^* \in C^2(\overline{\Omega^*})$ be positive functions satisfying \eqref{eq:mb}.  Suppose $u \in C^0(\overline{\Omega})$ is a $g$-convex Aleksandrov solution of \eqref{eq:gjef},\eqref{eq:2bvp} satisfying the following property: there is $x_0 \in \Omega$ and a support $g(\cdot,y_0,z_0)$ at $x_0$ such that $J_0:= [\inf_{\Omega}g(\cdot,y_0,z_0)-K_0\text{diam}(\Omega),\sup_{\Omega}g(\cdot,y_0,z_0)+K_0\text{diam}(\Omega)] \subset J$. Assume $\Omega^*$ is uniformly-$g^*$-convex with respect to points in $\Omega\times J_0$ and $\Omega$ is uniformly-$g$-convex with respect to points in $\Omega^*\times g^*(\Omega,\Omega^*,J_0)$.   Then $u \in C^3(\overline{\Omega})$. 
\end{theorem}

\section{Transformations}
\label{sec:transformations-1}
In this section we introduce transformations which leave the generating function close to $g(x,y,z) = x\cdot y -z $. Assume $x_0 \in U,y_0 \in V,u_0\in J$ and $h\geq0$ are given. For context, we usually have a $g$-convex function $u:\Omega \rightarrow \mathbf{R}$ and take $x_0 \in \Omega$ with $y_0 \in Yu(x_0),$ $u_0 = u(x_0)$ and shift the support to $g(\cdot,y_0,g^*(x_0,y_0,u_0+h))$. Without loss of generality $x_0,y_0,u_0=0$.  Set $z_h=g^*(0,0,h)$. After replacing $g$ by the function $(x,y,z) \mapsto g(x,y,z+g^*(0,0,h))$ we assume $g(0,0,0) = h$ so $g^*(0,0,h) = 0$. Furthermore by working in the coordinates $y' := E(0,0,0)y$  we have $E(0,0,0) = \text{Id}$. (We recall $E$ is the matrix from A2.)

\subsection*{Transformed coordinates}
\label{sec:transf-coord}

Define
\begin{align}
  q(x) &:= g_z(0,0,0)\left[\frac{g_y}{g_z}(x,0,0)-\frac{g_y}{g_z}(0,0,0)\right],\label{eq:xdef}\\
  p(y) &:= g_x(0,y,g^*(0,y,h))-g_x(0,0,0). \label{eq:ydef}
\end{align}
Conditions A1,A1$^*$, and A2  imply $x \mapsto q(x)$ and $y\mapsto p(y)$ are diffeomorphisms, so we may write $q=q(x)$, or $x=x(q)$ as necessary, similarly for $y$ and $p$. The Jacobian of the first transform is
\begin{equation}
  \label{eq:x-jac}
   \frac{\partial q_i}{\partial x_j} =  \frac{g_z(0,0,0)}{g_z(x,0,0)}E_{ji}(x,0,0).
\end{equation}
Because $g_z \det E \neq 0$ on $\overline{\Gamma}$ these transformations are non-degenerate. It is useful to introduce a quantity which quantifies this nondegeneracy and how far the generating function is from $g(x,y,z) = x\cdot y -z$. Put
\begin{align}
  &E^+ = \text{sup}_{\Gamma}\sup_{\xi \in \mathbf{S}^{n-1}}\vert E_{ij}\xi_i \vert  &&E^{-}=\text{inf}_{\Gamma}\inf_{\xi \in \mathbf{S}^{n-1}} \vert E_{ij}\xi_i \vert,\\
  &C_z = \sup_{\Gamma}\vert g_z\vert  && c_z = \inf_{\Gamma}\vert g_z\vert ,
\end{align}
and set $C_g = \max\{E^+,C_z,1/E^-,1/c_z\}$, where if $g(x,y,z) = x\cdot y -z$ we have $C_g = 1$. $C_g$ is used to quantify the effect of $g$ not being affine.  For example if $D\subset U$ and $D_q$ is its image under \eqref{eq:xdef}, then by \eqref{eq:x-jac} there holds
\begin{equation}
  \label{eq:xxbarest}
 C_g^{-3}\vert D \vert \leq \vert  D_q \vert \leq C_g^3\vert D\vert  . 
\end{equation}
A similar estimate, depending only on $C_g$, holds for the $y$ to $p$ transformation.

\subsection*{Generating function transformation}
\label{sec:gener-funct-nearly}

Set
\[ \tilde{g}(x,y,z) = \frac{g_z(0,0,0)}{g_z(x,0,0)}[g(x,y,g^*(0,y,h-z)) - g(x,0,0)] ,\]
and subsequently 
\[ \overline{g}(q,p,z) = \tilde{g}(x(q),y(p),z),\]
where $x,q$ and $y,p$ satisfy \eqref{eq:xdef} and \eqref{eq:ydef} respectively. As motivation note in the optimal transport case, where $g(x,y,z) = c(x,y)-z$ for $c$ the cost function, we have
\[ \tilde{g}(x,y,z) = [c(x,y)-c(0,y)] - [c(x,0)-c(0,0)]-z,\]
which is a frequently used transformation \cite{FKM13,LTW15,ChenWang16}. We note Jhaveri \cite{Jhaveri17} has made use of a different transformed generating function. 

The basic facts concerning $\overline{g}$ are summarized in the following lemma. We use the overline notation to denote quantities corresponding to $\overline{g}$.

\begin{lemma}\label{lem:transform_facts}
  Let $g$ be a generating function satisfying the LMP. Then:\\
  (1) $\overline{g}$ is a $C^2$ generating function satisfying the LMP.\\ 
  (2) A function $u$ is $g$-convex if and only if the corresponding function
  \begin{equation}
    \label{eq:u_transform}
         \overline{u}(q):= \frac{g_z(0,0,0)}{g_z(x(q),0,0)}[u(x(q)) - g(x(q),0,0)], 
  \end{equation}
    is $\overline{g}$-convex. Moreover, with $\overline{Y}$ defined for $\overline{g}$ as $Y$ was for $g$, we have  $y \in Yu(x)$ if and only if $p(y) \in \overline{Y}\overline{u}(q(x))$.\\
\end{lemma}
The proof is a direct calculation which we defer to Appendix \ref{sec:omitted-proofs}. Again, the effect of these transformations can be measured in terms of $C_g$. In particular for $\overline{u}$ as in Lemma \ref{lem:transform_facts} 
\begin{equation}
  \label{eq:u-cg-dep}
C_g^{-2}\vert \overline{u}(q)\vert  \leq  \vert u(x(q)) - g(x(q),0,0)\vert  \leq C_g^2\vert \overline{u}(q)\vert 
\end{equation}
We emphasize this because we will frequently consider $ \overline{u}(q)$ in place of $u(x(q)) - g(x(q),0,0)$ in certain estimates. The estimate \eqref{eq:u-cg-dep} implies this is acceptable as these quantities are comparable up to the constant $C_g$. 
Near the origin $\overline{g}$ is close to a plane in the sense of the following lemma. Such expansions are important for the regularity in optimal transport \cite{LTW10,LTW15}.  

\begin{lemma}\label{lem:identities}
Let $g$ be a generating function and $\overline{g}$ be the transformed generating function. There are $C^1$ functions $a^{(\alpha)}(q,p),b^{(\beta)}(q,p),f^{(\gamma)}(q,p,z)$ for $\alpha,\beta = 1,2,3$ and $\gamma=1,2$, which arise as Taylor series remainder terms, such that $\overline{g}$ satisfies the following identities
  \begin{align}
    \label{eq:x-exp}
    \overline{g}(q,p,z) &=q\cdot p-z +a^{(1)}_{ijk}(q,p)q_iq_jp_k+f^{(1)}(q,p,z)z,  \\
 \label{eq:y-exp}   \overline{g}(q,p,z) &=q\cdot p-z +a^{(2)}_{ijk}(q,p)q_ip_jp_k+f^{(1)}(q,p,z)z.
  \end{align}

  Here  $f^{(1)}$ satisfies the inequalities 
  \begin{align}
    \label{eq:f-est-1}
    -C^- \leq -1+f^{(1)}(q,p,z) &\leq -C^+,\\
   \text{ and } \label{eq:f-est-2}  \vert f^{(1)}(q,p,z)\vert &\leq C\vert q \vert
  \end{align}
  for positive constants $C^{\pm}$ depending only on $C_g$ and $C$ depending, in addition, on $sup \vert g_{xz} \vert$. Furthermore $f^{(1)}$ satisfies the equality
  \begin{align}
    \label{eq:f-exp}
    f^{(1)}(q,p,z) = b^{(1)}_{ij}q_iq_j+b^{(2)}_{ij}q_ip_j+b^{(3)}_{ij}p_ip_j + f^{(2)}(q,p,z)z.
  \end{align}
\end{lemma}
\begin{proof}
  First, write
  \begin{align}
    \label{eq:exp1}
    \overline{g}(q,p,z) = \overline{g}(q,p,0)+\int_0^1 \overline{g}_z(q,p,t  z)z \ dt.
  \end{align}
  We take $f^{(1)}(q,p,z) = (1+\int_0^1 \overline{g}_z(q,p,t  z) \ dt)z$. Then $f^{(1)}$ satisfies \eqref{eq:f-est-1} because by a direct calculation $- C_g^4\leq \overline{g}_z \leq - C_g^{-4} $. Note here, and throughout this proof, we rely on the calculations \eqref{eq:gstar-deriv}. Next, because $\overline{g}_z(0,p,z) = -1$, \eqref{eq:f-est-2} is  the Lipschitz continuity of this quantity, as guaranteed by A0. To obtain \eqref{eq:f-exp} we expand with a Taylor series in $z$, then in $q,p$ and obtain 
  \begin{align}
    \label{eq:6}
    \overline{g}_z&(q,p,z) = \overline{g}_z(q,p,0)z+\frac{1}{2}\overline{g}_{zz}(q,p,\tau z)z^2\\
    &= \overline{g}_z(0,0,0)z+\overline{g}_{q_i,z}(0,0,0)q_iz+\overline{g}_{p_i,z}(0,0,0)p_iz+g_{zz}(q,p,\tau z)z^2\label{eq:full-exp}\\
\nonumber    &\quad+z[\overline{g}_{q_i,p_j,z}(q_t,p_t,0)q_ip_j+\overline{g}_{q_iq_j,z}(q_t,p_t,0)q_iq_j+\overline{g}_{p_ip_j,z}(q_t,p_t,0)p_ip_j],
  \end{align}
  where $t,\tau \in (0,1)$, $q_t = tq$ and similarly for $p_t$.
 Using $\overline{g}_z(0,p,0) = \overline{g}_z(q,0,0) = -1,$ and subsequently  $\overline{g}_{q_i,z}(0,0,0) = \overline{g}_{p_i,z}(0,0,0) = 0$, \eqref{eq:full-exp} implies \eqref{eq:f-exp}. Whilst we've not explicitly used it, the integral form of the remainder term implies the coefficients of $q_iq_j,q_ip_j,p_ip_j$ are $C^1$. 

  We're left to deal with the term $\overline{g}(q,p,0)$ in \eqref{eq:exp1}. Set $\tilde{c}(x,y) := \tilde{g}(x,y,0)$ and $c(q,p) = \tilde{c}(x(q),y(p))$. This suggestive notation indicates all the following calculations are based on the optimal transport case \cite{FKM13,ChenWang16,LTW15}.  By direct computation $\tilde{c}$ satisfies $\tilde{c}_x(0,y) = p$ and $\tilde{c}_y(x,0) = q$ along with $\tilde{c}(x,0) \equiv 0, $ $\tilde{c}(0,y) \equiv 0$ and $\tilde{c}_{i,j}(0,0) = \delta_{ij}$. It's these identities that justify our inclusion of the $g_z$ terms in the definitions of $q,p,\overline{g}$. They imply
  \begin{align}
 \label{eq:r1}   &c(0,0) = 0 &&c_{q}(0,0) = 0 &&&c_{p}(0,0) = 0\\
\label{eq:r2}    &c_{p_i,q_j}(0,0) = \delta_{ij} && c_{qq,p}(q,0) = 0 &&& c_{q,pp}(0,p) = 0.
  \end{align}
  Thus via a Taylor series
  \begin{align}
    \label{eq:t1}    c(q,p) &= c(0,p) + c_{q_i}(0,p)q_i + c_{q_i,q_j}(tq,p)q_iq_j,
  \end{align}
  As usual $t,\tau \in (0,1)$ (we're about to use $\tau$). More Taylor series and we obtain 
  \begin{align}
    \label{eq:t2}  c_{q_i}(0,p) &= c_{q_i}(0,0)+c_{q_i,p_j}(0,0)p_j + c_{q_i,p_jp_k}(0,\tau p)p_jp_k =  p_i\\
  \label{eq:t3}   c_{q_iq_j}(tq,p) &=  c_{q_iq_j}(tq,0) + c_{q_iq_j,p_k}(tq,\tau p)p_k = a_{ijk}^{(1)}p_k.
  \end{align}
  The integral form of the remainder term implies $a^{(1)}$ is $C^1$. Combining \eqref{eq:t1}-\eqref{eq:t3} yields \eqref{eq:x-exp}. Similar calculations imply \eqref{eq:y-exp}. 
\end{proof}

\begin{remark}
  When $g$ is $C^4$ an additional term in the Taylor series \eqref{eq:t3} yields
  \begin{align}
    \label{eq:xy-exp}
    \overline{g}(q,p,z) &=q\cdot p-z +a^{(3)}_{ij,kl}(q,p)q_iq_jp_kp_l+f^{(1)}(q,p,z)z,
  \end{align}
  though we don't use this expansion here. 
\end{remark}

\section{$g$-cones}
\label{sec:g-cones-subd}
Cones are a basic tool for studying the Monge--Amp\`ere equation. A similar class of functions was introduced in the optimal transport setting by Figalli, Kim, and McCann \cite[Section 6.2]{FKM13}. The defining feature of this so-called $c$-cone is that its $Y$ mapping is concentrated at a point. The generalization to $g$-cones is due to Guillen and Kitagawa \cite{GuillenKitagawa17}. In each case we want estimates for the $Y$-mapping of the generalised cone in terms of the base and height of the generalised cone.  

Let $u:\Omega \rightarrow \mathbf{R}$ be a $g$-convex function. Assume $x_0 \in \Omega, y_0 \in V$ are given and $u_0:=u(x_0)$. For $h > 0$ small set $z_h = g^*(x_0,y_0,u_0+h)$ and assume
\begin{equation}
  \label{eq:D-def}
  D := \{x \in \Omega;u(x) < g(x,y_0,z_h) \} \subset\subset \Omega.
\end{equation}
We define the $g$-cone with vertex $(x_0,u_0)$ and base $\{(x,g(x,y_0,z_h)); x \in \partial D\}$ by  
\begin{align}
  \label{eq:2vdef}
  \vee(x) &= \sup\{\phi_y(x) := g(x,y,g^*(x_0,y,u_0)) ;  \phi_y(x) \leq g(x,y_0,z_h) \text{ on }\partial D\}. \nonumber
\end{align}
When we need to emphasize dependencies we  include them as a subscript, e.g. $\vee_{D,h}$ if $x_0,y_0,u_0$ are clear from context. The expression for $\vee$ is well defined for any $D \subset U$.  However requiring  $D$ satisfy \eqref{eq:D-def} ensures $\vee = g(\cdot,y_0,z_h)$ on $\partial D$.

Our goal is to estimate $Y \vee_{D,h}(x_0)$ in terms of $D$ and $h$. As in Section \ref{sec:transformations-1} we assume, without loss of generality, that $x_0,y_0,u_0,z_h = 0$. Using Lemma  \ref{lem:transform_facts}(2) it suffices to work in the coordinates given by \eqref{eq:xdef}, \eqref{eq:ydef} and estimate the $\overline{Y}$ mapping of
\[ \overline{\vee}(q) := \frac{g_z(0,0,0)}{g_z(x(q),0,0)}[\vee(x(q)) - g(x(q),0,0)].\]
 By direct calculation we see $\overline{\vee}$ is the $\overline{g}$-cone with base $\partial D_q\times \{0\}$ and vertex $(0,-h)$ (recall $D_q$ is the image of $D$ in the $q$ coordinates). Thus
\begin{equation}
  \label{eq:simple}
   \overline{\vee}(q) = \sup\{\phi_p(q) := \overline{g}(q,p,h); \phi_p \leq 0 \text{ on }D_q\},
\end{equation}
with $D_q$ convex and $\overline{g}$ satisfying Lemma \ref{lem:identities}. To simplify notation we revert to $x,y,g,\vee$.

We compare $\vee$ with a cone of the same base and height. The cone with vertex $(0,-h)$ and base $\partial D \times \{0\}$ is $K:D \rightarrow \mathbf{R}$ defined by
\begin{equation}
  \label{eq:k-def}
   K(x)  = \sup\{l_p(x):= p\cdot x - h ; p \in \mathbf{R}^n \text{ and }l_p \leq 0 \text{ on }\partial D \}.
\end{equation}

\subsection{Upper bounds for $Y \vee (0)$}
\label{sec:upper-bounds-partial}

\begin{lemma}\label{lem:upper_est}
  Suppose $g$ is a generating function satisfying LMP and the identities in Lemma \ref{lem:identities}. Let  $D$ be  a convex domain containing 0. Suppose $\vee,K$ are as defined in \eqref{eq:simple},\eqref{eq:k-def} respectively. There is $d_0,h_0>0$ such that if $\text{diam}(D) \leq d_0$ and $h \leq h_0$ then
  \begin{equation}
    \label{eq:y-upper}
      Y  \vee(0) \subset 2\partial K(0). 
  \end{equation}
Where $d_0,h_0$ depend on $\Vert g \Vert_{C^3}$ and $2\partial K(0) = \partial (2K)(0)$. 
\end{lemma}

\begin{proof}
  We prove the transformed generating function satisfies
  \begin{equation}
    \label{eq:almost-affine}
     g(x,y,h) \geq \frac{3}{4} x \cdot y - \frac{3h}{2},
  \end{equation}
  for $\vert x \vert,h$ sufficiently small and $x \cdot y ,h>0$,\footnote{we note if $x \cdot y$ or $h<0$ \eqref{eq:almost-affine} holds with $3/4$ replaced by $5/4$ or $3/2$ replaced by $1/2$. }; \eqref{eq:y-upper} is a straightforward consequence. Indeed, take $y \in Y\vee(0)$ and suppose $y \notin 2\partial K(0)$, that is $x \cdot y > 2h$ for some boundary point $x \in \partial D$. By \eqref{eq:almost-affine} $g(x,y,h) > 0$ and so $g(\cdot,y,h)$ can not be a support of $\vee$.

Take $y \in V$ and rotate so $y = (0,\dots,0,y_n)$. Let $x = (x_1,\dots,x_n) \in D$ satisfy $x_ny_n >0$ and set $x'=(x_1,\dots,x_{n-1},0)$. We assume, for now, $g(x',y,0) \geq 0$ (we'll see this is a consequence of LMP). Now, \eqref{eq:x-exp} implies
  \begin{equation}
   g_{x_n}(x_\tau,y,0)x_n \geq x_ny_n - K\vert x \vert  x_n y_n ,\label{eq:x-est-cone}
 \end{equation}
  for $x_\tau= \tau x + (1-\tau)x'$ and $\tau \in [0,1]$ where $K$ depends on $\Vert g \Vert_{C^3}$. By \eqref{eq:f-est-2} and a choice of $\text{diam}(D)$ small
  \begin{align*}
   g(x,y,h) \geq g(x,y,0) - \frac{3h}{2}.
\end{align*}
A Taylor series for $h(t):= g(tx+(1-t)x',y,0)$, our assumption $g(x',y,0) \geq 0$, and  \eqref{eq:x-est-cone} imply
\begin{align*}
g(x,y,h) &\geq g(x',y,0)+g_{x_n}(x_\tau,y,0)x_n-\frac{3}{2}h \geq x_ny_n(1- K \vert x \vert ) - \frac{3}{2}h.
  \end{align*}
  Choosing $\text{diam}(D)$ small to ensure $K \vert x \vert  \leq 1/4$ we obtain \eqref{eq:almost-affine}.

  To conclude we show $g(x',y,0) \geq 0$. Since $x' \cdot y = 0$ it suffices to show whenever $x \cdot y > 0$ then $g(x,y,0) \ge 0$ and use continuity.  Note if $x \cdot y > 0$ the expression \eqref{eq:x-exp} implies
  \[ g(tx,y,0) > 0 \text{ and } g(-tx,y,0) < 0,\]
  for $t > 0$ sufficiently small. If $g(x,y,0) < 0$ then the $g$-convexity of the section $\{g(\cdot,y,0) < 0 = g(\cdot,0,0)\}$ with respect to $0,0$ (which is just convexity), is violated. So as required $ g(x,y,0) \ge 0$.
\end{proof}

\subsection{Lower bounds for $\partial \vee (0)$}
\label{sec:lower-bounds-partial}

The estimates in the other direction are formulated differently. As motivation consider the rectangle
  \begin{equation}
   R = \{ x \in \mathbf{R}^n; -b_i \leq x_i \leq a_i\},\label{eq:R}
 \end{equation}
  for $a_i,b_i>0$,  and the cone $K$ with base $R$ and vertex $(0,-h)$. Then $\partial K(0)$ contains the points $he_i/a_i,-he_i/b_i$. Thus with
   \begin{equation}
   R^* := \{x \in \mathbf{R}^n; -b_i^{-1} \leq x_i \leq a_i^{-1}\}.\label{eq:Rstar}
 \end{equation}
we have
 \begin{align}
  \label{eq:k-est-1} \partial K(0) &\supset C_n h R^*\\
  \label{eq:k-est-2}    \vert \partial K(0) \vert  &\geq C_n h^n  \prod_{i=1}^n\left(\frac{1}{b_i}+\frac{1}{a_i}\right). 
 \end{align}
 Next we decrease the base of the cone: consider a domain $D$ with $0 \in D \subset R$ and the cone $K_D$ with base $D$ and vertex $(0,-h)$.
 Because $\partial K(0) \subset \partial K_D(0)$, \eqref{eq:k-est-1} and \eqref{eq:k-est-2} hold with $K$ replaced by $K_D$. This motivates the following result. 

\begin{lemma}\label{lem:lower_est}
  Suppose $g$ is a generating function satisfying LMP and the identities in Lemma \ref{lem:identities}. Let $D$ be a convex domain with $0 \in D \subset R$. Let $\vee$ be given by \eqref{eq:simple}. There is $d_0,h_0 > 0$  such that if $\text{diam}(D) \leq d_0$ and $h \leq h_0$ then \eqref{eq:k-est-1} and \eqref{eq:k-est-2} hold with $\vee$ in place of $K$. The quantities $d_0,h_0$ depend on $\text{diam}(V), \Vert g \Vert_{C^3}$, 
\end{lemma}

\begin{proof}[Proof. (Lemma \ref{lem:lower_est})]
 Rather than showing $\partial \vee(0)$ contains the points $Che_i/a_i$ and $-Che_i/b_i$ for $i=1,\dots,n$ we will, instead,  show $\partial {\vee}(0)$ contains points close to these points. That is, we show for for some $\kappa \geq 1/4$ and $q := \kappa h e_n/a_n$ there is $p \in \partial \vee(0)$  satisfying
  \begin{equation}
    \label{eq:pqrel}
       \vert p-q \vert  \leq \frac{1}{16} \vert q \vert .
  \end{equation}
 Our proof also applies to $\kappa h e_i/a_i$ and  $-\kappa h e_i/b_i$ for $i=1,\dots,n$, so $ChR^* \subset \partial  \vee(0)$.

  To begin, choose $\hat{x}$ realizing $\hat{x}_n = \sup\{x_n ; x = (x_1,\dots,x_n) \in D\}$. We see, by taking a limit of the $\phi_y$ used in \eqref{eq:simple} to define $\vee$, that there is $\hat{y}$ for which $g(\cdot,\hat{y},h)$ supports $\vee$ at $\hat{x}$ and $0$. In particular, since $\vee = 0$ at $\hat{x}$ and is less than or equal to $0$ on $\partial D$ we have, for  $\hat{y}$ appropriately chosen and some $\beta \geq 0$,
  \begin{equation}
    \label{eq:onxn}
       g_x(\hat{x},\hat{y},h) = \beta e_n. 
  \end{equation}
We'll prove that $p = g_x(0,\hat{y},h)$ and $q = (p\cdot e_n)e_n$ satisfy \eqref{eq:pqrel}. 

Choose $d^*$ so that $g(d^*e_n,\hat{y},h) = 0$. We claim $d^*\leq a_n$. Indeed
\[S:= \{x; g(x,\hat{y},h) < 0 = g(x,0,0)\}  \] is $g$-convex with respect to $(0,0)$, that is, convex. Furthermore since $g_x(\hat{x},\hat{y},h) = \beta e_n$ and $g(\hat{x},\hat{y},h) = 0$, the plane  $P:=\{x; x_n = \hat{x}_n\}$ supports $S$. Thus, since $S$ contains $0$ and lies on one side of $P$, $S \subset \{x;x_n \leq \hat{x}_n\}$ and $d^* \leq a_n$.

  Now \eqref{eq:x-exp} with \eqref{eq:f-exp} implies
  \begin{equation}
    \label{eq:est}
         \vert g_x(x,\hat{y},h) - g_x(0,\hat{y},h) \vert  \leq C( \vert x \vert +h) \vert  \vert \hat{y} \vert  + Kh(h+ \vert x \vert ),
      \end{equation}
      where $C,K$ depend on $\Vert g \Vert_{C^3},$ $\text{diam}(V)$ and we now assume $\text{diam}(D) \leq1$. By a Taylor series we obtain for some $\tau \in (0,1)$
  \begin{align}
  \nonumber  h &= g(d^*e_n,\hat{y},h) - g(0,\hat{y},h)\\
  \nonumber    &= d^*g_{x_n}(\tau d^*e_n,\hat{y},h)\\
   \label{eq:d-ineq} &\leq d^*  \vert g_x(0,\hat{y},h) \vert +Cd^*(d^*+h) \vert \hat{y} \vert  + Kh(h+d^*).
  \end{align}
  To estimate $ \vert \hat{y} \vert $ in terms of $ \vert g_x(0,\hat{y},h) \vert $ write
  \begin{equation}
    \label{eq:haty_est}
        \vert \hat{y} \vert  =  \vert g_x(0,\hat{y},0) \vert  \leq  \vert g_x(0,\hat{y},h) \vert  +  \vert g_{xz}(0,\hat{y},\tau h) \vert h. 
  \end{equation}
  Combining \eqref{eq:d-ineq} and \eqref{eq:haty_est} we have
  \[ h \leq d^* \vert g_x(0,\hat{y},h) \vert [1+Cd^*(d^*+h)] + Kh(h+d^*).\]
  We choose $\text{diam}(D)$ and $h$ small to ensure both $(1+Cd^*(d^*+h)) \leq 3/2$ and $K(h+d^*) \leq 1/4$. Combining with $d^* \leq a_n$ yields
  \begin{equation}
    \label{eq:subdiff_est}
   \frac{h}{2a_n} \leq    \vert g_x(0,\hat{y},h) \vert . 
  \end{equation}

  Using, once again, \eqref{eq:est} (this time with $x = \hat{x}$) and \eqref{eq:haty_est} we have
  \[   \vert g_x(\hat{x},\hat{y},h) - g_x(0,\hat{y},h) \vert  \leq C( \vert \hat{x} \vert +h) \vert  \vert g_x(0,\hat{y},h) \vert +Kh( \vert \hat{x} \vert +h) . \]
  Dividing through by $ \vert g_x(0,\hat{y},h) \vert $, using \eqref{eq:subdiff_est} and choosing $h, \vert \hat{x} \vert $ sufficiently small we can ensure 
  \[ \left \vert \frac{g_x(\hat{x},\hat{y},h)}{ \vert g_x(0,\hat{y},h) \vert }- \frac{g_x(0,\hat{y},h)}{ \vert g_x(0,\hat{y},h) \vert }\right \vert  \leq 1/16.\]
  The first vector lies on the $e_n$ axis (recall \eqref{eq:onxn}). Thus the unit vector $\frac{g_x(0,\hat{y},h)}{ \vert g_x(0,\hat{y},h) \vert }$, and consequently $g_x(0,\hat{y},h)$ make angle $\theta$ with the $e_n$ axis for $\theta$ satisfying $\sin(\theta) \leq 1/16$, i.e. $\theta \leq 1/8$. This, with \eqref{eq:subdiff_est} implies both
  \[ e_n \cdot g_x(0,\hat{y},h) = \cos (\theta) \vert g_x(0,\hat{y},h) \vert  \geq \frac{h}{4a_n},\]
  and
  \[  \vert g_x(0,\hat{y},h)-(e_n \cdot g_x(0,\hat{y},h))e_n  \vert  \leq \sin (\theta)  \vert g_x(0,\hat{y},h) \vert  \leq \frac{1}{16} \vert (e_n \cdot g_x(0,\hat{y},h))e_n \vert ,\]
  which is \eqref{eq:pqrel}. \end{proof}

We have a more precise estimate when $x_0$ is close to the boundary. We make use of the minimum ellipsoid (see \cite[Section 2.1]{LiuWang15}) and the following (specific case of a) lemma due to Figalli, Kim, and McCann \cite{FKM13}. 
\begin{lemma}\cite[Lemma 6.9]{FKM13}\label{lem:fkm}
  Let $D \subset \mathbf{R}^n$ be a convex domain. Assume $D$ contains a ``vertical'' line segment $\{(x',t_0+t); x' \in \mathbf{R}^{n-1} , t \in [0,d]\}$ of length $d$. Let
  \[D': = \{(x_1,\dots,x_{n-1},0); x = (x_1,\dots,x_n) \in D\}\] be the projection of $D$ onto $\mathbf{R}^{n-1}$. There is $C>0$ depending only on $n$ such that
  \[  \vert D \vert \geq  Cd \mathcal{H}^{n-1}(D'),\]
  where $\mathcal{H}^{n-1}$ is the $n-1$ dimensional Hausdorff measure. 
\end{lemma}

Lemma in hand, we prove the estimate close to the boundary using a similar proof to that of Figalli, Kim, and McCann. 
\begin{lemma}\label{lem:lower_est_close}   Suppose $g$ is a generating function satisfying LMP and the identities in Lemma \ref{lem:identities}. Let $D$ be a convex domain with $0 \in D$. Let $\vee$ be given by \eqref{eq:simple}.  Assume $0$ is close to the boundary, in the sense that there is a unit vector $\nu$ and positive $d$ such that
  \begin{equation}
   \sup_{x \in D}\langle x,\nu \rangle = \epsilon d, \label{eq:close_cond}
 \end{equation}
 and in addition $D$ contains a line segment of length $d$ parallel to $\nu$. There is $C,d_0,h_0 > 0$ depending on $ \Vert g \Vert_{C^3}$, such that if $\text{diam}(D) \leq d_0$ and $h \leq h_0$ then
  \[ h^n \leq  C\epsilon  \vert \partial \vee(0) \vert  \vert D \vert . \]
\end{lemma}
\begin{proof}
  We assume, without loss of generality that $\nu = e_n$. Let $D'$ be as in Lemma \ref{lem:fkm}. Then up to a choice of the remaining coordinates we assume the minimum ellipsoid of $D'$ (as a subset of $\mathbf{R}^{n-1}$) is
  \[ E':= \{x' = (x_1,\dots,x_{n-1},0); \sum_{i=1}^{n-1}\left(\frac{x_i-\overline{x}_i}{b_i/2}\right)^2 \leq 1\},\] for some $\overline{x} \in D'$. Then $\mathcal{H}^{n-1}(D') \geq C_n b_1 \dots b_{n-1}$ and
  \[ D \subset [-b_1,b_1]\times \dots \times [-b_{n-1},b_{n-1}]\times [-K,\epsilon d],\]
  for some $K>0$. Then Lemma \ref{lem:lower_est}, specifically \eqref{eq:k-est-2}, implies
  \begin{equation}
    \label{eq:cb}
       \vert \partial \vee(0) \vert  \geq C_n h^n \frac{1}{\epsilon d b_1 \dots b_{n-1}}. 
  \end{equation}
  On the other hand Lemma \ref{lem:fkm} implies
  $ \vert D \vert  \geq C_n d  \vert D' \vert  \geq C_n d b_1 \dots b_{n-1} $, which combined with \eqref{eq:cb} completes the proof. 
\end{proof}

\section{Uniform estimates}
\label{sec:uniform-estimates}

In this section we consider $g$-convex Aleksandrov solutions of
\begin{align}
\label{eq:caff-style} \lambda&\leq \det DYu \leq \Lambda   \text{ in }D,\\
\label{eq:dir}  u &= g(\cdot,y_0,z_0) \text{ on }\partial D.
\end{align}

Here $\lambda,\Lambda$ are positive constants and $D$ (being a section) is necessarily $g$-convex with respect to $y_0,z_0$. Our goal is to estimate $ \vert u-g(\cdot,y_0,z_0) \vert $ in D. By using the $g$-cone estimates our proofs are direct extensions of those in the Monge--Amp\`ere case. In this section we assume $\text{diam}(D)$ and $h := \sup  \vert u-g(\cdot,y_0,z_0) \vert $ are sufficiently small as required by Lemmas \ref{lem:upper_est}, \ref{lem:lower_est} and \ref{lem:lower_est_close}.

\begin{theorem}\label{thm:upper_uniform}
Assume $g$ is a generating function satisfying LMP.  Assume $u$ is a $g$-convex solution of \eqref{eq:caff-style}, \eqref{eq:dir}. There is $C>0$ depending only on $\Lambda,n,C_g$ such that
  \begin{equation}
    \label{eq:upper_uniform}
       \sup_{D} \vert u(\cdot)-g(\cdot,y_0,z_0) \vert ^n \leq C \vert D \vert ^2.
     \end{equation}
\end{theorem}
\begin{proof} Fix any $x_0 \in D$, after translating assumed to be 0. It suffices to obtain \eqref{eq:upper_uniform} after applying the transformations in Section \ref{sec:transformations-1}. Estimate \eqref{eq:upper_uniform} will  hold for the original function and coordinates with appropriate inclusion of the constant $C_g$. Thus we assume $D$ is convex, $g$ satisfies the identities in Lemma \ref{lem:identities} and $u \equiv 0$ on $\partial D$.  Let the minimum ellipsoid of $D$ be given by
  \[  E := \big\{x; \sum\frac{(x_i-\overline{x}_{i})}{a_i^2} \leq 1\big\}, \]
  for some $\overline{x} \in D$. Note  $D \subset R := \{-2a_i\leq x_i \leq 2a_i\}$. Set $R^* = \{-(2a_i)^{-1}\leq x_i \leq (2a_i)^{-1}\}$ and note $\vert R\vert ^* \geq C_n\vert D\vert ^{-1}$. Thus by Lemma \ref{lem:lower_est} the $g$-cone $\vee$ with vertex $(0,u(0))$ and base $\partial D \times \{0\}$ satisfies $C \vert D \vert \vert \partial \vee(0) \vert  \geq  \vert u(0) \vert ^n$. Furthermore $Y\vee(0) = Y(0,\vee(0),\partial \vee(0))$ so $\vert Y\vee(0) \vert  \geq C \vert \partial \vee(0) \vert $ for $C$ depending on $C_g$. Thus 
  \begin{equation}
    \label{eq:est-alm}
        \vert D \vert \vert Y\vee(0)  \vert  \geq C \vert u(0) \vert ^n  .
  \end{equation}

  The comparison principle (\cite[Lemma 4.4]{Trudinger14}) implies $Y\vee(0)\subset Yu(D)$ so \eqref{eq:est-alm} along with $\vert Yu(D)\vert \leq \Lambda \vert D \vert $ implies the result. 
\end{proof}

If instead we use Lemma \ref{lem:lower_est_close} we obtain the following:

\begin{theorem}\label{thm:upper_uniform_close} Assume $g$ satisfies LMP and the identities in Lemma \ref{lem:identities} and that $D$ is a convex domain. Let $u$ satisfy \eqref{eq:caff-style} and \eqref{eq:dir}. Suppose further that $x_0 \in D$ is a point close to the boundary in the sense that there is a unit vector $\nu$ and $\epsilon,d>0$ with
  \[ \sup_{x \in D}\langle x-x_0,\nu \rangle = \epsilon d \]
  and $D$ contains a line segment of length $d$ parallel to $\nu$. There is $C>0$ independent of $u$ such that
  
  \[  \vert u(x_0)-g(x_0,y_0,z_0) \vert ^n \leq C\epsilon \vert D \vert ^2.\]
\end{theorem}

\begin{remark}\label{rem:ineq}
  For Theorems \ref{thm:upper_uniform} and \ref{thm:upper_uniform_close} it suffices that $u \geq g(\cdot,y_0,z_0)$ on $\partial D$. In this case we apply the above proofs to $D' = \{ u \leq g(\cdot,y_0,z_0)\} \subset D$ which is assumed nonempty.  Similarly for the lower bound, Theorem \ref{thm:lower_uniform}, it suffices to have $u \leq g(\cdot,y_0,z_0)$ on $\partial D$.
\end{remark}

The lower bound uses a lemma of Guillen and Kitagawa's \cite{GuillenKitagawa17}. We use it only in the case of the transformed generating function.  This simplifies the proof, which we give after the proof of Theorem \ref{thm:lower_uniform}. 

\begin{lemma}\cite[Lemma 6.1]{GuillenKitagawa17}\label{lem:gk}
 Let $g$ be the transformed generating function from Lemma \ref{lem:transform_facts}. Suppose $u$ is $g$-convex and $D := \{u \leq 0\}$ has $0$ as the centre of its minimum ellipsoid. Set $h = \sup_{D} \vert u \vert $. There are constants $C,K > 0$ depending only $\Vert g \Vert_{C^3}$ such that
  \[ Yu\left(\frac{1}{K}D\right) \subset Y\vee_{D,Ch}(0),\]
  where $\vee_{D,Ch}$ is the $g$-cone with base $D \times \{0\}$ and vertex $(0,-Ch)$.
\end{lemma}

\begin{theorem}\label{thm:lower_uniform}
  Assume $g$ is a generating function satisfying LMP and $u$ is a $g$-convex solution of \eqref{eq:caff-style},\eqref{eq:dir}. There is a constant $C$ depending on $\lambda,n,\Vert g \Vert_{C^3}$ such that
  \[ C \vert D \vert ^2 \leq \sup_{D} \vert u(\cdot)-g(\cdot,y_0,z_0) \vert ^n.\]
\end{theorem}
\begin{proof}
  We assume $y_0,z_0 = 0$ then pick any $x_0 \in D$ and make the change
  \eqref{eq:xdef} so that $D$ is convex. After translation we assume the minimum ellipsoid for $D$ is centred on $0$. We make the remaining changes in Section \ref{sec:transformations-1} so  $g$ satisfies Lemma \ref{lem:identities} and $u = 0$ on $\partial D$. 

Let  $\vee = \vee_{\frac{1}{K}D,Ch}$ be the $g$-cone from Lemma \ref{lem:gk} and note 
  \[ \frac{\lambda}{K^n} \vert D \vert  \leq \left \vert Yu\left(\frac{1}{K}D\right)\right \vert  \leq  \vert Y\vee(0) \vert .\]  Let $\hat{K}$ be the classical cone with the same base and vertex as $\vee$. By Lemma \ref{lem:upper_est}
  \begin{equation}
    \label{eq:lower-1}
     \frac{\lambda}{K^n} \vert D \vert  \leq C \vert \partial \hat{K}(0) \vert . 
  \end{equation}
The estimate 
\begin{equation}
  \label{eq:lower-2}
   \vert \partial \hat{K}(0) \vert  \leq \frac{Ch^n}{ \vert D \vert },
\end{equation}
follows from standard convex geometry making use of the minimum ellipsoid. Combining \eqref{eq:lower-1} and \eqref{eq:lower-2} completes the proof. 
\end{proof}

 The proof of Lemma \ref{lem:gk} uses the Loeper maximum principle via the quantitative quasiconvexity interpretation of Guillen and Kitagawa. More precisely if the Loeper maximum principle holds then so does the following statement: Let $x_0,x_1 \in U$, $y_0,y_1 \in V$ and $u_0 \in J$ be given. Let $\{x_\theta\}_{\theta \in [0,1]}$ denote the $g$-segment (in our coordinates, line segment) from $x_0$ to $x_1$ with respect to $y_0,z_0 = g^*(x_0,y_0,u_0)$ and set $z_1 = g^*(x_0,y_1,u_0)$. Then there is $M$ depending only on $\Vert g \Vert_{C^3}$ such that
  \begin{align}
    g(x_\theta,y_1,z_1) - g(x_\theta,y_0,z_0) \leq M \theta [g(x_1,y_1,z_1) - g(x_1,y_0,z_0)]_{+}.\label{eq:gqq}
  \end{align}
  We outline a short proof based on recent work of Loeper and Trudinger \cite{LoeperTrudinger21} in the Appendix, Lemma \ref{lemma:lmp-gqq} (see also \cite{Jeong21}). 

\begin{proof}[Proof. (Lemma \ref{lem:gk})] 
  We assume $K$ has been fixed small, to be chosen in the proof, and show there is $C$ such that
  \begin{equation}
    \label{eq:gk-main}
      Yu\left(\frac{1}{K}D\right) \subset Y\vee_{D,Ch}(0).
    \end{equation}
  To this end, fix $x \in \frac{1}{K}D$ and $y \in Yu(x)$. Let $g_0$ be the corresponding support
  \[ g_0(\cdot):= g(\cdot,y,g^*(x,y,u(x))) =g(\cdot,y, \vert g_0(0) \vert ),\]
  where the second equality is because the transformed generating function satisfies $g^*(0,y,u) = -u$. 
  To prove \eqref{eq:gk-main} it suffices to show there is $C$ (independent of $x,y$) such that
  \begin{equation}
    \label{eq:gk-equiv}
        \vert g_0(0) \vert  \leq Ch.
  \end{equation}
  For in this case the function $g(\cdot,y,Ch)$ passes through the vertex of $\vee = \vee_{D,Ch}$ and lies below $g_0$ so is nonpositive on $D$. Thus $y \in Y\vee(0)$.

  So let's prove \eqref{eq:gk-equiv}. By a Taylor series for $h(t) = g(x,y,t\vert g_0(0) \vert)$ there is $C^+,C^-$ (depending only on $C_g$) such that for any $x' \in D$
  \begin{align}
     \label{eq:g1-1} g(x',y,0) &\geq g_0(x') + C^- \vert g_0(0) \vert \\
   \label{eq:g1-2} g(x',y,0)  &\leq g_0(x') + C^+ \vert g_0(0) \vert  \leq C^+ \vert g_0(0) \vert .   
  \end{align}
 
  Now let $\{x_\theta\}_{\theta \in [0,1]}$ denote the $g$-segment with respect to $0,0$ that starts at $0$, passes through $x$, and hits $\partial D$ at some $x_1$. Because $x \in \frac{1}{K}D$ there is $\theta' \leq 1/K$ with $x_{\theta'} = x$. So \eqref{eq:gqq} implies
  \begin{align}
       g(x,y,0) \leq M\theta' [g(x_1,y,0)]_+, \label{eq:pos-part}
  \end{align}
  where we've used \eqref{eq:gqq} with $y_1=y$ and $y_0,z_0,u_0 = 0$ so $g(\cdot,y_0,z_0) = 0$.
  If $g(x_1,y,0) \leq 0$, then  $g(x,y,0) \leq 0$ and we obtain \eqref{eq:gk-equiv} from \eqref{eq:g1-1} with $x' = x$ (because $g_0$ is a support at $x$ we have $g_0(x)=u(x) \geq -h$).

  Otherwise, combine \eqref{eq:pos-part} with \eqref{eq:g1-1} on the left hand side and \eqref{eq:g1-2} on the right hand side to obtain
  \[ g_0(x) + C^- \vert g_0(0) \vert  \leq M\theta' C^+ \vert g_0(0) \vert . \]
  Recalling $g_0(x) \geq -h$ and choosing $K$ large to ensure $M\theta' C^+ \leq C^-/2$ completes the proof. \end{proof}

\section{Strict convexity assuming a $g$-convex containing domain}
\label{sec:strict-conv}
In this section we prove the strict $g$-convexity, that is Theorem \ref{thm:strict-convexity}, by adapting Chen and Wang's work from the optimal transport case.

\begin{proof} [Proof (Theorem \ref{thm:strict-convexity}).]
  We extend $u$ to $\tilde{u}$ defined on $\overline{U}$ as 
\[ \tilde{u}(x) := \sup\{g(x,y_0,z_0); g(\cdot,y_0,z_0) \text{ is a $g$-support of $u$ in }\Omega\}.\]
This extension is equal to $u$ on $\Omega$ and satisfies
\begin{align}
  \label{eq:gjeext}  \lambda\chi_{\overline{\Omega}} \leq \det DY\tilde{u} &\leq \Lambda\chi_{\overline{\Omega}} ,\\
  \label{eq:2bvpext}Y\tilde{u}(\overline{U}) &= \overline{\Omega^*}.
\end{align}
We assume work with $\tilde{u}$, though  keep the notation $u$.

  For a contradiction we suppose there is a support $g(\cdot,y_0,z_0)$ such that
  \[ G:= \{x \in  \overline{U}; u(x) = g(x,y_0,z_0)\},\]
  contains more than one point in $\Omega$. The first step of the proof is to show that, after the coordinate transform with respect to $y_0,z_0$, any extreme point of the (now convex) set $G$ is in $ \partial U$. The second step is to choose a particular extreme point and obtain a contradiction from the fact that it is in $\partial U$. 

  \textit{Step 1. Extreme points cannot be in the interior}\\
  Without loss of generality $y_0,z_0=0$. Applying the transformation \eqref{eq:xdef} of the $x$-coordinates we have that $G$ and $U$ are convex. Assume, for a contradiction, there is an extreme point of $G$, without loss of generality 0, which is an interior point of $U$. After transforming the $y$ coordinates and generating function as in Section \ref{sec:transformations-1} we have
   \[ G  = \{ x;u(x) = 0 = g(x,0,0) \},\]
  and $u\geq0$. Choose a plane $P$ that supports $G$ at $0$ and rotate so that 
   \begin{align}
P = \{x_{1} = 0\},\quad\quad  G \cap P = \{0\},\quad\quad G \subset  \{x; x_1 \leq 0\} \quad \text{ and }\quad -a e_1  \in G,
   \end{align}
   for some $a >0$.
 Set
  \[ G^h := \{u < g(\cdot,0,-h)\}.\]
  Because $\Omega$ is open there is $x \in G\cap \Omega$ with $B:=B_r(x) \subset \Omega$ for sufficiently small $r>0$. In particular $\det DYu \geq \lambda$ on $B$. Choose $r,h$ small enough to ensure Theorem \ref{thm:lower_uniform} holds on $B\cap G^h $ (recalling Remark \ref{rem:ineq}). Then
  \begin{equation}
    \label{eq:ghest1}
       C  \vert G^h\cap B \vert ^{2/n} \leq h.
  \end{equation}
The $g$-convexity of $G^h$ (with respect to $0,-h$) implies
  \begin{equation}
  C  \vert G^h \vert  \leq   \vert G^h\cap B \vert  .\label{eq:convballest}
  \end{equation}
  where $C$ depends on $r$, $g$, and some upper bound $h_0 \geq h$ \footnote{For details see  \cite[pg. 101]{Figalli17} for the convex case and \cite{RankinThesis} for the $g$-convex case.}. By \eqref{eq:ghest1} and \eqref{eq:convballest}
  \begin{equation}
    \label{eq:to_cont}
    C  \vert G^h \vert ^{2/n} \leq h.
  \end{equation}

To obtain a contradiction by Theorem \ref{thm:upper_uniform_close} we consider section that behave like $\{u < t(x_1+a) \}$ in the convex case. By \eqref{eq:y-exp} for $t$, sufficiently small,
 \begin{equation}
   \label{eq:gstar-rep}
      g(-ae_1,te_1,0) = -at + O(t^2) < 0.
 \end{equation}
  This implies $ g(\cdot,te_1,0) < g(\cdot,te_1,g^*(-ae_1,te_1,0))$, since the second function is $0$ at $-ae_1$, i.e greater than the first. Subsequently
  \[ D_t := \{u < g(\cdot,te_1,g^*(-ae_1,te_1,0))\},\]
  satisfies $0 \in D_t$, and $-ae_1 \in \partial D_t$. On the other hand, the convergence $g(\cdot,te_1,g^*(-ae_1,te_1,0)) \rightarrow 0$ as $t \rightarrow 0$ implies
  \begin{equation}
    \label{eq:lim-cond}
     a^+_t := \sup\{x_1 = x\cdot e_1; x \in D_t\} \rightarrow 0.
   \end{equation}
   Finally by the expansions in Lemma \ref{lem:identities} and \eqref{eq:gstar-rep}
   \begin{equation}
     \label{eq:0est}  \vert u(0) - g(0,te_1,g^*(-ae_1,te_1,0)) \vert  \geq Ct.
   \end{equation}

 To apply Theorem \ref{thm:upper_uniform_close} we need a convex section. Transform the $x$ coordinates according to \eqref{eq:xdef} with respect to $y_t=te_1,z_t=g^*(-a e_1,y_t,0)$. Take the line segment joining the images under this transformation of $-ae_1$ and $0$. Its length is greater than $Ca$ for a constant depending on $g$. One of the supporting planes orthogonal to this line segment converges to the plane $P = \{x_{1} = 0\}$ as $t \rightarrow 0$. The other remains a distance of at least $Ca$ from the image of $0$.   Thus Theorem \ref{thm:upper_uniform_close} implies, for some $\epsilon_t \rightarrow 0$, that
  \begin{equation}
    \label{eq:aona}
     t \leq C\epsilon_t \vert D_t \vert ^{2/n},
  \end{equation}
  where we've used \eqref{eq:0est}.
  Finally, again by \eqref{eq:y-exp},\eqref{eq:f-exp} and \eqref{eq:gstar-rep} we obtain $D_t \subset G^{Ct}$ for some $C>0$ and $t$ sufficiently small. Thus \eqref{eq:aona} with $t = h/C$ contradicts \eqref{eq:to_cont} and we conclude any extreme point of $G$ lies on $\partial U$.\\ \\
  \textit{Step 2. An extreme point that cannot be on the boundary}\\
  \textit{Step 2a. Coordinate transform}
  The argument in this step requires transforming the coordinates and generating function with respect to different points so we are explicit with the details. We begin with the original coordinates, generating function, and contact set $G = \{u \equiv g(\cdot,y_0,z_0)\}$ which we assume contains more than one point in $\Omega$. Without loss of generality $y_0,z_0= 0$. Introduce the coordinates
  \begin{equation}
    \label{eq:x-exp-1}
       \tilde{x} = \frac{-g_y}{g_z}(x,0,0).
  \end{equation}
  Pick any $x_1 \in G$ and denote its image under \eqref{eq:x-exp-1} by $\tilde{x_1}$. Define
  \[ \overline{y} = E^{-1}(x_1,0,0)[g_x(x_1,y,g^*(x_1,y,u(x_1))) - g_x(x_1,0,0)].\]
  The $E^{-1}$ factor implies, via a Taylor series in $y$, that
  \begin{equation}
    \label{eq:y-o2}
    \overline{y} = y + O( \vert y \vert ^2).
  \end{equation}
The images of $\Omega,\Omega^*$ in these coordinates are denoted by\footnote{In step 2a the overline notation is used for coordinates not closures.} $\tilde{\Omega},\overline{\Omega^*}$. Both are convex and $0 \in \overline{\Omega^*}$. By a rotation, which is applied to $\tilde{x}$ and $\overline{y}$, we assume $t_0e_1 \in \overline{\Omega^*}$ for some small $t_0$. By convexity $\overline{\Omega^*}$ contains a cone
  \begin{equation}
    \label{eq:cone-arg}
       \mathcal{C} := \{t \xi; 0 < t \leq t_0, \xi \in B_r'\},
  \end{equation}
  where $B_r '$ is a geodesic ball in the hemisphere $\mathbf{S}^{n-1} \cap \{x_1 > 0\}$ centered on $(1,0,\dots,0)$ (Figure \ref{fig:transform1}).

  Now choose a specific extreme point of $\tilde{G}$ as follows: Take the paraboloid
\[ \tilde{P}_M = \{\tilde{x} = (\tilde{x}_1,\dots,\tilde{x}_n) ; \tilde{x}_1 = -\epsilon(\tilde{x}_2^2+\dots+\tilde{x}_n^2) + M\}\]
for large $M$ and small $\epsilon$. Decrease $M$ until the paraboloid first touches $\tilde{G}$, necessarily at an extreme point $\tilde{x_0}$. We take $\tilde{P}$ as the tangent plane to $\tilde{P}_M$ at $\tilde{x_0}$ and note $\tilde{P}$ supports $\tilde{G}$ at $\tilde{x_0}$. Provided $\epsilon$ is sufficiently small (depending on  $\text{diam}(G)$) the normal to $\tilde{P}$ is in $B_r'$ (Figure \ref{fig:transform1}).

  Our final coordinate transform is
  \begin{equation}
    \label{eq:x-fin}
      \overline{x} \mapsto -g_z(x_0,y_0,z_0)[\tilde{x}-\tilde{x_0}].
  \end{equation}
  which is a dilation and translation (but no rotation) of the $\tilde{x}$ coordinates. In these coordinates the image of $U$, denoted $\overline{U}$, is convex, and provided $\epsilon$ was chosen small depending also on $ C_g $ (but importantly independent of $x_0$), we have the normal to $\overline{P}$, the image of $\tilde{P}$ under \eqref{eq:x-fin}, at $0$ is still in $B_r'$. Set
  \begin{align*}
    \tilde{g}(x,y,z) &= \frac{g_z(x_0,0,0)}{g_z(x,0,0)}[g(x,y,g^*(x_0,y,u(x_0)-z)) - g(x,0,0)],
   \end{align*}
  and define $\overline{g}(\overline{x},\overline{y},z) = \tilde{g}(x,y,z)$ for $\overline{x},\overline{y}$ the image of $x,y$. By a Taylor series in $y,z$, using $\tilde{g}_y(x,0,0) = \overline{x}, \tilde{g}_z(x,0,0)= -1$ and \eqref{eq:y-o2}, we have
  \begin{equation}
    \label{eq:genexp-2}
       \overline{g}(\overline{x},\overline{y},z) = \overline{x} \cdot \overline{y} -z+ O( \vert \overline{y} \vert ^2)+O( \vert \overline{y} \vert  \vert z \vert ) +O( \vert z \vert ^2)
  \end{equation}

  Finally we rotate the $\overline{x}$ and $\overline{y}$ coordinates so that the supporting plane $\overline{P}$ becomes
  \[\overline{P} = \{\overline{x}=(\overline{x}_1,\dots,\overline{x}_n); \overline{x}_1 = 0\},\] and $\overline{G} \subset \{ \overline{x} ; \overline{x}_1 \leq 0\}$. Because the normal to $\overline{P}$ was in $B_r'$ after this rotation we still have $\overline{y_0}':=t_0e_1 \in \overline{\Omega^*}$. This is by inclusion of the cone \eqref{eq:cone-arg} (Figure \ref{fig:transform1}).  Finally pick any $\overline{x_0'}\in \overline{G}$ such that $\overline{x_0'} \cdot \overline{y_0 '}\neq 0$. We are in exactly the setting to use the transformation (4.7) from \cite{ChenWang16} which preserves that $\overline{g}$ has the form \eqref{eq:genexp-2} and after which we have
  \begin{align*}
      &\overline{G}  \subset \{\overline{x}; \overline{x}_1 \leq 0\},  &&\overline{G} \cap \{\overline{x}_1 = 0\} = 0,\\
    &-a e_1  \in \overline{G}, && be_1 \in \overline{\Omega^*}. 
  \end{align*}
  Step 2b takes place in this setting. For ease of notation we drop the overline.

  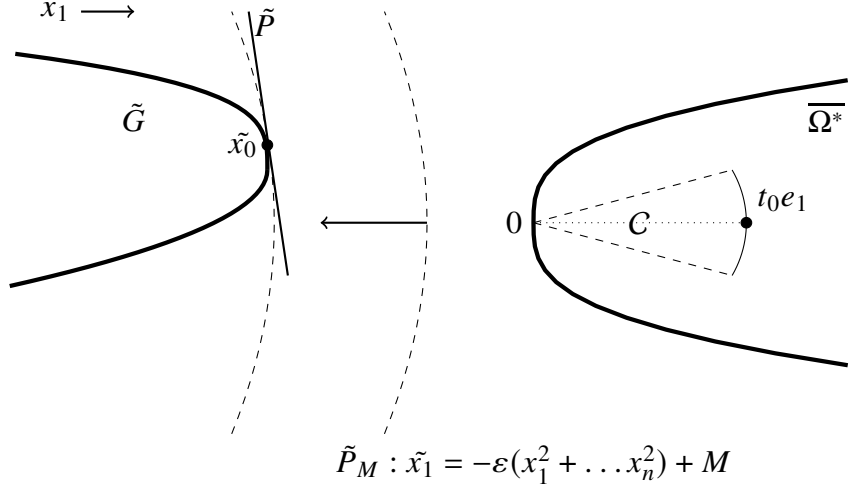
\begin{figure}
  \begin{tikzpicture}[scale=0.7]
    % Omega *
    %\draw[help lines] (0,0) grid (20,10);
    \node [left] at (1.5,9) {$x_1$};
    \draw [->,thick] (1.5,9) -- (2.5,9);
    \node (c) at (15.5,7) {$\overline{\Omega^*}$};
    \draw[ultra thick, domain=2.3:7.7, variable=\y] plot ({10+0.3*(\y-5)*(\y-5)*abs((\y-5))}, \y);
    \node [above right] at (14,5) {$t_0e_1$};
    \draw[fill] (14,5) circle [radius=0.1];
    \draw [dotted] (10,5) -- (14,5) ;
    \draw[domain=4:6, variable=\y] plot ({12+sqrt(4-pow((\y-5),2))}, \y);
    \draw [dashed] (10,5) -- (13.73205,4) ;
    \draw [dashed] (10,5) -- (13.73205,6) ;
    \node  at (12,5) {$\mathcal{C}$};
    \node  [left] at (10,5) {$0$};

    % Omega
    \draw[ultra thick, domain=6:8.2, variable=\y] plot ({5-0.3*pow((\y-6),3.5)}, \y)  ;
    \draw[ultra thick, domain=3.8:6, variable=\y] plot ({5-pow((\y-6),2)}, \y)  ;

    \draw[dashed, domain=1:9, variable=\y] plot ({8-0.05*pow((\y-5),2)}, \y)  ;
    \draw[dashed, domain=1:9, variable=\y] plot ({5.13-0.05*pow((\y-5),2)}, \y)  ;
    \draw[fill] (5.0,6.47196) circle [radius=0.1];
    \node  at (2.5,7) {$\tilde{G}$};
    \node [below] at (10,1) {$\tilde{P}_M: \tilde{x_1} = -\epsilon(x_1^2+\dots x_n^2)+M$};
    \draw [->,thick] (8,5) -- (6,5);
    \node  [left] at (5.021666,6.47196) {$\tilde{x_0}$};
    \draw[thick,domain=4:9, variable=\y] plot ({-0.147196*\y+5.974303}, \y)  ;
     \node at (4.95,8.8) {$\tilde{P}$};
  \end{tikzpicture}
  \caption{The choice of $\overline{y}$ coordinates implies $\overline{\Omega^*}$ contains the cone $\mathcal{C}$. By choosing $\epsilon$ sufficiently small a normal to the paraboloid at the extreme point lies in the cone. Thus after a rotation so $\tilde{P}$ becomes $\overline{P} = \{\overline{x};\overline{x}_1=0\}$, $\overline{\Omega^*}$ still contains $t_0e_1$. }\label{fig:transform1}
\end{figure}

\textit{Step 2b. Obtaining the contradiction in this setting}\\
  We see $F_\tau := G\cap \{x_1 \geq -\tau\}$ decreases to $\{0\}$ as $\tau \rightarrow 0$. Because $\Omega$ lies a positive distance from $0 \in \partial U$, for $\tau$ fixed sufficiently small $\text{dist}(\Omega,F_\tau) >  0$. On the other hand, provided in addition, $\tau < -a/2$ the set
  \begin{equation}
    \label{eq:new-set}
     \{ u < g(\cdot,te_1,g^*(-\tau e_1,te_1,u(-\tau e_1))\} ,   
  \end{equation}
 decreases to $F_\tau$ as $t\rightarrow0$. This is because by calculations similar to part 1 
  \[ g(x,te_1,g^*(-\tau e_1,te_1,u(-\tau e_1)) = (x_1+\tau)t+O(t^2).\]
 Thus for $t$ small the set in \eqref{eq:new-set} is disjoint from $\overline{\Omega}$. So for all $x \in \overline{\Omega}$,
\begin{equation}
  \label{eq:final_cont}
  u(x) > g(x,te_1,g^*(-\tau e_1,te_1,u(-\tau e_1)).
\end{equation}

On the other hand $te_1 \in \Omega^*$ for $t$ small. Thus there is $x_2 \in \overline{\Omega}$ with $te_1 \in Yu(x_2)$. So
\begin{equation}
  \label{eq:gives_cont}
   u(x) \geq g(x,te_1,g^*(x_2,te_1,u(x_2))),
\end{equation}
for all $x \in U$. For a contradiction evaluate \eqref{eq:gives_cont} at $x = -\tau e_1$, then apply
\[g(x_2,te_1,g^*(-\tau e_1,te_1,\cdot)) \]
to both sides. The resulting inequality contradicts \eqref{eq:final_cont} with $x = x_2$ and thereby completes the proof.
\end{proof}

\begin{remark}\label{rem:unif-dom}
 A corollary of Theorem \ref{thm:strict-convexity} is that when $\Omega$ is uniformly $g$-convex with respect to $u$ and $\overline{\Omega} \subset U$ then $u$ is strictly $g$-convex: no convexity condition is needed on $U$. This is immediate; for $\epsilon$ sufficiently small $\Omega_\epsilon := \{x ; \text{dist}(x,\Omega) < \epsilon\}$ is uniformly $g$-convex with respect to $u$ and strictly contains  $\Omega$.
\end{remark}

It would be desirable to prove the strict convexity under the extension of the hypotheses of Figalli, Kim, McCann \cite{FKM13}. These hypotheses are strict $c$-convexity of both domains, which is stronger than here --- however they only require the cost function be defined on $\overline{\Omega}\times\overline{\Omega^*}$. Unfortunately when one attempts to extend their crucial Theorem 5.1 a $g_{i,j,z}$ term appears in the analogue of their equation (5.2) and prohibits a similar proof.

\section{$C^1$ differentiability of strictly $g$-convex solutions}
\label{sec:furth-cons-strict}
Once the strict convexity has been proved we obtain the $C^1$ differentiability using similar techniques.
\begin{theorem}
  Assume $g$ satisfies LMP and $\lambda,\Lambda > 0$. Suppose $u:\Omega \rightarrow \mathbf{R}$ is a strictly $g$-convex Aleksandrov solution of $ \lambda \leq \det DYu \leq \Lambda.$  Then $u \in C^1(\Omega)$.
\end{theorem}
\begin{proof}
  Suppose for a contradiction at some $x_0$, assumed to be $0$, $\partial u(0)$ contains more than one point. Let $p_0$ be an extreme point of $\partial u(0)$ with $y_0 := Y(0,u(0),p_0)$ and $g(\cdot,y_0,z_0)$ the corresponding support. Without loss of generality $u(0),y_0 = 0$. Fix $h>0$ small, and apply the transformations in Section \ref{sec:transformations-1} so $g$ satisfies \eqref{eq:x-exp} and $u(0)=-h$. We set $  G^h = \{ u < 0 = g(\cdot,0,0)\},$  which is a section and thus convex.
  
  After these transformations, by \eqref{eq:x-exp} and \eqref{eq:f-exp}, $p=g_x(0,0,h)$ is an extreme point of $\partial u(0)$ satisfying $ \vert p \vert  = O(h^2)$. Thus after subtracting $-h+p\cdot x$ from both $u$ and the generating function we have
  \begin{equation}
    \label{eq:c1-subset}
       \{u < h/2\} \subset  G^h = \{u < h - p \cdot x \} \subset \{u < 3h/2\}.
  \end{equation}
  We've used that by the strict convexity and an initial choice of the section small  $ \vert x \vert $ is as small as desired. Moreover, $0$ is now an extreme point of $\partial u(0)$ and we assume, after a rotation, for some small $a>0$
  \begin{align}
  \nonumber  \{ p; p_1 \geq 0\} &\supset \partial u(0) \\
 \label{eq:subdiff2}   ae_1 &\in \partial u(0) ,
  \end{align}
  for some small $a>0$. This implies  $u(-te_1) = o(t)$. Thus $ \{ u < h/2\},$
  contains $-R(h)he_1$ for some positive function $R$ satisfying $R(h) \rightarrow \infty$ as $h \rightarrow 0$ (for details, see \cite[Lemma A.4]{RankinThesis}).

  We recall inequality \ref{eq:almost-affine}, so that for $x$ in a small neighbourhood of $0$ we have
  \[ u \geq g(x,Y(0,ae_1,u(0)),h) \geq \frac{x_1a}{2} -Ch, \]
  where we've used $Y(0,ae_1,u(0)) = ae_1+O(h)$. So for $h$ small, the second subset relation in \eqref{eq:c1-subset} implies $\sup_{G^h}x_1 \leq Ch/a$.

  Thus Theorem \ref{thm:upper_uniform_close} implies
  \[ h \leq \frac{C}{R(h)} \vert G^h \vert ^{2/n}.\]
  On the other hand Theorem \ref{thm:lower_uniform} yields $ \vert G^h \vert ^{2/n} \leq Ch$, which is a contradiction as $h \rightarrow 0$.
 \end{proof}

\section{Global regularity}
\label{sec:global-regularity}

In this section we use the strict $g$-convexity to prove the global regularity of Aleksandrov solutions of the second boundary value problem.

To prove Theorem \ref{thm:global-regularity} it suffices to prove there is $v \in C^3(\overline{\Omega})$ solving \eqref{eq:gjef} subject to \eqref{eq:2bvp} and satisfying $v(x_0) = u(x_0)$ for some $x_0 \in \Omega.$ Here's why. Our strict $g$-convexity result, Remark \ref{rem:unif-dom}, implies the Aleksandrov solution is strictly $g$-convex. Then by recent work of Trudinger \cite[Theorem 3.4]{Trudinger20}, $u \in C^3(\Omega)$. Thus $u,v$ are regular enough to apply the uniqueness result, \cite[Theorem 1.1]{Rankin2020}, and conclude $u = v$ so $u \in C^3(\overline{\Omega})$.

The existence of the desired $v$ is proved by modifying Jiang and Trudinger's \cite{JiangTrudinger18} global existence result so as to construct a solution taking a prescribed value at a given point. For the Monge--Amp\`ere and optimal transport cases this is trivial --- just add or subtract a constant. In our case we modify Jiang and Trudinger's use of degree theory to obtain a solution that is close to the prescribed value at a given point. Taking a limit gives the desired globally smooth solution. We note apart from a slightly more restrictive height condition (which ensures the constructed function doesn't leave $J$), our hypotheses are those from the existence theory \cite[Theorem 1.1]{JiangTrudinger18}.

For the degree theory we use that under the hypotheses of Theorem \ref{thm:global-regularity} the second boundary value problem can be rewritten as a Monge--Amp\`ere type equation \cite{Trudinger14} coupled with a uniformly oblique boundary condition \cite{JiangTrudinger18}. That is $C^4(\Omega) \cap C^3(\overline{\Omega})$ solutions of \eqref{eq:gjef} subject to \eqref{eq:2bvp} solve
\begin{align}
  \det[D^2u - A(\cdot,u,Du)] &= B(\cdot,u,Du),\text{ in }\Omega\\
  G(\cdot,u,Du) &= 0 \text{ on }\partial \Omega
\end{align}
where
\begin{align*}
  A(\cdot,u,Du) &= g_{ij}(\cdot,Y(\cdot,u,Du),Z(\cdot,u,Du))\\
  B(\cdot,u,Du) &= \det E \frac{f(\cdot)}{f^*(Y(\cdot,u,Du))},
\end{align*}
and $G(x,u,p)$ satisfies
\[ G_p(x,u,Du) \cdot \gamma \geq c_0 >0 \text{ on }\partial \Omega,\]
for $\gamma$ the outer unit normal to $\partial \Omega$ and $c_0$ independent of $u$ \cite{JiangTrudinger18}. Written in this way we introduce the following additional assumptions used by Jiang and Trudinger for the existence theory, and used here for the regularity theory.

\textbf{A4w. } The matrix $A$ satisfies
\[ D_uA_{ij}(x,u,p)\xi_i\xi_j \geq 0,\]
for all $(x,u,p) \in \mathcal{U}$ and $\xi \in \mathbf{R}^n$.

\textbf{A5. } Assume $u:\Omega \rightarrow \mathbf{R}$ is a $g$-convex function. There is $K_0$ depending on $g$ and $Yu(\Omega)$ such that for all $x \in \Omega,y \in Yu(x)$, and $z = g^*(x,y,u(x))$, there holds $\vert g_x(x,y,z)\vert \leq K_0$.

The degree theory serves as a high-powered version of the method of continuity. As in the method of continuity we need apriori estimates and a smooth function solving a problem in the same homotopy class. These are provided by the following results from the literature. 

\begin{theorem}\cite[Theorem 3.1]{JiangTrudinger14}\label{thm:c2-est}
  Assume $g$ is a $C^4$ generating function satisfying LMP. Assume $u\in C^4(\Omega) \cap C^3(\overline{\Omega})$ is a $g$-convex solution of \eqref{eq:gjef} subject to \eqref{eq:2bvp}. Then provided $f\in C^2(\overline{\Omega}),f^*\in C^2(\overline{\Omega^*})$ satisfy the mass balance condition \eqref{eq:mb} and $\Omega,\Omega^*$ are respectively uniformly $g/g^*$-convex with respect to $u$, there is $C>0$, depending on $A,B,\Omega,\Omega^*$ and $\Vert u\Vert_{C^1(\Omega)}$ such that
  \[ \Vert u \Vert_{C^2}(\Omega) \leq C.\]  
\end{theorem}

\begin{remark}\label{rem:higher-ests}
  If $\Vert u \Vert_{C^1(\Omega)}$ is controlled independently of $u$, then the $C^2$ estimate is independent of $u$.  Subsequently when $u$ has higher order derivatives, estimates for these follow by the elliptic theory. More precisely $C^{2,\alpha}(\overline{\Omega})$ estimates from \cite[Theorem 1]{LiebermanTrudinger86} then $C^{4,\alpha}(\overline{\Omega})$ estimates by the linear theory. The use of the linear theory is standard and outlined in \cite[Lemma 7.7]{RankinThesis}. For the $C^{4,\alpha}$ estimates, as in \cite{JiangTrudinger18}, we  smooth our domains so they are $C^5$. However uniform $C^{3,\alpha}$ estimates are independent of this smoothing. 
\end{remark}

\begin{lemma}\cite[Lemma 2.3]{JiangTrudinger18}\label{lem:approx-exist}
 Fix $x_0 \in \Omega,y_0 \in \Omega^*$, $z_0 \in I_{x_0,y_0}$ and $g$-convex $g_0(\cdot):=g(\cdot,y_0,z_0)$. Assume $\Omega,\Omega^*$ are uniformly $g/g^*$-convex with respect to $g_0$. Define
  \[ \Omega_\delta = \{x \in U; \text{dist}(x,\Omega) < \delta\}.\]
  Then for all $0<\epsilon<<\delta$ sufficiently small there is a smooth uniformly $g$-convex $u_\epsilon \in C^\infty(\overline{\Omega_{\delta}})$ and a smooth uniformly $g^*$-convex domain  $\Omega^*_{\delta,\epsilon}$ satisfying the following convergence properties
  \begin{align*}
    &\Omega^*_{\delta,\epsilon} \rightarrow  \Omega^*_{\delta,0} \text{ in the Hausdorff distance as }\epsilon \rightarrow 0\\
     &\Omega^*_{\delta,0} \rightarrow  \Omega^* \text{ in the Hausdorff distance as }\delta \rightarrow 0
  \end{align*}
  and
  \begin{equation}
    \label{eq:u-ep-height}
      g(\cdot,y_0,z_0) \leq u_\epsilon \leq  g(\cdot,y_0,z_0)+C\epsilon, 
  \end{equation}
  with
  \[ Yu_\epsilon(\Omega_{\delta}) = \Omega^*_{\delta,\epsilon}.\]
  The uniform $g^*$-convexity of is $\Omega^*_{\delta,\epsilon}$ is with respect to $u_\epsilon,$ and $\Omega_\delta$ is uniformly $g$-convex with respect to $u_\epsilon$. 
\end{lemma}

The following lemma, in particular condition \eqref{eq:f-app-3} is what allows us to prove the global regularity. We let $x_0,g(\cdot,y_0,z_0)$ be as in Theorem \ref{thm:global-regularity} and use the notation from Lemma \ref{lem:approx-exist}.
\begin{lemma}\label{lem:approx-prob}
  Assume the hypothesis of Theorem \ref{thm:global-regularity}. There exists $v \in C^{3,\alpha}(\overline{\Omega}_\delta)$ depending on $\delta,\epsilon$ and satisfying
  \begin{align}
 \label{eq:f-app-1}   \det DY(\cdot,v,Dv) &= \frac{f_\delta(\cdot)}{f_\epsilon^*(Y(\cdot,v,Dv))}\text{ in }\Omega_\delta\\
 \label{eq:f-app-2}   Yv(\Omega_\delta) &= \Omega^*_{\delta,\epsilon}\\
    v(x_0) &= u_\epsilon(x_0) \label{eq:f-app-3}
  \end{align}
\end{lemma}
\begin{proof}
  Within this proof we denote $f,f^*,\Omega,\Omega^*$, and $u_0$ by $f_\delta,f^*_\epsilon,$ $\Omega_\delta,\Omega^*_{\delta,\epsilon},$ and $u_\epsilon$. Fix a smooth cutoff function $\eta$ for the unit ball, i.e. $\eta > 0$ on $B_1(0)$ and $\eta = 0$ on $\overline{B_1}^c$. For $a<<1$ set $\eta_a(x) =  a^{4}\eta((x-x_0)/a)$ where the $a^{4}$ term ensures $\Vert \eta_a \Vert_{C^4} \leq C(\eta)$.

  For $t \in [0,1]$ and $\tau$ to be fixed large we consider the family of problems
\begin{align}
  \label{eq:homotopy}  f^*(Y(\cdot,v,Dv))&\det DY(\cdot,v,Dv) = e^{[(1-t)\tau+\eta_a(\cdot)](v-u_0)}\big[tf(\cdot)\\
                                   \nonumber                     &\quad+(1-t)f^*(Y(\cdot,u_0,Du_0))\det DY(\cdot,u_0,Du_0)\big]\\
\label{eq:2bvp-approx}  Yv(\Omega) &= \Omega^*. 
\end{align}  
We restrict our attention to $g$-convex solutions $v \in C^{4,\alpha}(\overline{\Omega})$. For $t=0$ the problem is solved by $u_0$. We aim to show, using the degree theory of Li, Liu, and Nguyen \cite{LLN17}, that this problem has a solution for $t=1$. The required $C^{4,\alpha}(\overline{\Omega})$ bounds hold by Remark \ref{rem:higher-ests} provided we obtain a $C^1$ estimate. Then the degree theory is applied as follows. By \cite[Theorem 1 (p2)]{LLN17} the problems \eqref{eq:homotopy} and \eqref{eq:2bvp-approx} have the same degree for $t=0$ and $t=1$. Then by \cite[Corollary 2.1 (a)]{LLN17} to show a solution exists for $t=1$ it suffices to show the problem at $t=0$ has non-zero degree. Finally by a combination of Corollary 2.1 (d) and Theorem 1 (p3) from the same paper, the problem has non-zero degree at $t=0$ provided both the problem for $t=0$ has a unique solution and the linearized problem at $t=0$ (linearized about $u_0$) is uniquely solvable. 

\textit{Step 1. a-priori $\Vert v \Vert_{C^1}$ estimates} We show using the mass balance condition that any solution of \eqref{eq:homotopy} subject to \eqref{eq:2bvp-approx} intersects $u_0$ in $\Omega$. Furthermore, for $t=1$ this intersection occurs in $B_a(x_0)$. This will allow us to send $a \rightarrow 0$ and obtain \eqref{eq:f-app-3}. Assume to the contrary $v > u_0$ on $\Omega$. The proof is similar if $v < u_0$ on $\Omega$. Because $v > u_0$ we have $ e^{[(1-t)\tau+\eta_a(\cdot)](v-u_0)} > 1$. Now \eqref{eq:homotopy} and \eqref{eq:2bvp-approx} along with mass balance and the change of variables formula yield the following contradiction
  \begin{align*}
    \int_{\Omega^*}f^* &= \int_{\Omega} f^*(Y(\cdot,v,Dv))\det DY(\cdot,v,Dv)  \\
                         &> \int_{\Omega}tf(\cdot) + (1-t)f^*(Y(\cdot,u_0,Du_0))\det DY(\cdot,u_0,Du_0)\\
    &= \int_{\Omega^*}f^*. 
  \end{align*}
  For $t=1$ we obtain the same contradiction if $v > u_0$ on $B_a$ for in this case 
  \begin{align*}
    f^*(Y(\cdot,v,Dv))\det DY(\cdot,v,Dv) &> f \text{ on }B_a(x_0)\\
    f^*(Y(\cdot,v,Dv))\det DY(\cdot,v,Dv) &= f \text{ on }\Omega\setminus\overline{B_a(x_0)}.
  \end{align*}
  An estimate $\sup_{\Omega} \vert Dv \vert \leq C$ follows by A5 provided $v(\overline{\Omega}) \subset J$. This inclusion follows by using contact point $v(x') = u_0(x')$, the inequality \eqref{eq:u-ep-height}, and the condition on $g(\cdot,y_0,z_0)$ in the statement of the Theorem \ref{thm:global-regularity}. Therefore $\Vert v \Vert_{C^1(\overline{\Omega})} \leq C$.

  \textit{Step 2. Unique solvability of the linear and nonlinear problem for $t=0.$}\\ First note the nonlinear problem has a unique solution by \cite[Theorem 1.1]{Rankin2020}. The linearisation at $t=0$ is a problem of the form
  \begin{align}
  \label{eq:r2:lin-fin}  a^{ij}D_{ij}v +b^iD_iv - cv &= 0 \text{ in }\Omega\\
  \nonumber  \alpha v + \beta \cdot Dv &= 0 \text{ on }\partial \Omega,
  \end{align}
  where $a^{ij},b^i,\alpha,\beta$ depend on $u_0$ and $g$, the equation is uniformly elliptic, $\beta$ is a strictly oblique vector field (by \cite{JiangTrudinger18}, see also \cite[Theorem 7.4]{RankinThesis}) and, provided $\tau$ is sufficiently large, $c > \tau/2$. We show this problem has the unique solution $v=0$ (thus, by the Fredholm Alternative, also proving the existence of a solution). If not, assume $v$ is a solution positive at some point  in $\overline{\Omega}$. Take a defining function $\phi$ for $\Omega$, satisfying $D\phi = \gamma$ (the outer unit normal), and set $w = e^{-\kappa\phi}v$ for large $\kappa$ to be chosen.  If the positive maximum of $w$ occurs at $x_1 \in \partial \Omega$ then by the obliqueness
  \[ 0 \leq \beta \cdot Dw(x_1) = e^{-\kappa \phi}[\beta \cdot Dv - \kappa (\beta \cdot \gamma) v].\]
  With the linearised boundary condition we obtain
  \[ \kappa (\beta \cdot \gamma)v(x_1) \leq \beta \cdot Dv(x_1) = - \alpha v(x_1), \]
  a contradiction for $\kappa$ sufficiently large depending only on $\alpha$,$c_0$. Thus $w$ attains its positive maximum at $x_1 \in \Omega$. At this point
  \begin{align}
 \nonumber   0 &\geq e^{\kappa \phi}a^{ij}D_{ij}w(x_1)  \\
 \label{eq:r2:lin-ineq}     &= a^{ij}D_{ij}v - 2\kappa a^{ij}D_i\phi D_j v - \kappa v a^{ij}D_{ij}\phi  + \kappa ^2 v a^{ij}D_i\phi D_j \phi. 
  \end{align}
  At an interior maximum $Dw(x_1) = 0$ so $D_iv(x_1) = \kappa D_i \phi(x_1) v(x_1)$. Thus we can eliminate $Dv$ terms in \eqref{eq:r2:lin-ineq}. In combination with \eqref{eq:r2:lin-fin} and $c > \tau/2$ we obtain
  \[ 0 \geq e^{\kappa \phi}a^{ij}D_{ij}w(x_1) \geq C(\tau - C_1)v(x_1),\]
  for $C_1$ independent of $\tau$. This is a contradiction for $\tau$ sufficiently large.  

  \textit{Conclusion.} By the degree theory there is $v \in C^{4,\alpha}(\overline{\Omega})$ solving \eqref{eq:homotopy} subject to \eqref{eq:2bvp-approx} for $t=1$. Our $C^{3,\alpha}$ estimates are independent of $a$ and the domain smoothing. Thus we consider a sequence of $a_k \rightarrow 0$ and the corresponding solutions at $t=1$ denoted $v_{a_k} \in C^{3,\alpha}(\overline{\Omega})$. By Arzela--Ascoli we have uniform convergence to some $v\in C^{3,\alpha}(\Omega)$ solving  \eqref{eq:homotopy} for $t=1$ and \eqref{eq:2bvp-approx}. Noting there is $x_k \in B_{a_k}(x_0)$ with $v_{a_k}(x_k) = u_0(x_{k})$ we conclude by the uniform convergence that the limiting function $v$ satisfies $v(x_0) = u_0(x_0)$.     
\end{proof}

To complete the proof of Theorem \ref{thm:global-regularity} take the solution of \eqref{eq:f-app-1}-\eqref{eq:f-app-3} and note the $C^{3,\alpha}$ bounds are independent of the parameters $\epsilon,\delta$ (provided these parameters are initially fixed small). Send first $\epsilon$ then $\delta$ to $0$ and obtain a $C^3(\overline{\Omega})$ solution of \eqref{eq:gje} subject to \eqref{eq:2bvp}  satisfying $u(x_0)=v(x_0)$, thereby completing the proof of Theorem \ref{thm:global-regularity} outlined at the start of this section. 

\appendix

\section{Omitted proofs}
\label{sec:omitted-proofs}

\begin{proof}[Proof. Theorem \ref{lem:transform_facts}]
Point (2) is immediate by direct calculation (verify directly \eqref{eq:conv-def1} and \eqref{eq:conv-def2}). The first point is also by direct calculation, but more involved. To begin, note
  \begin{align}   
  \label{eq:newgx}  \tilde{g}_x(x,y,z) &= \frac{-g_{xz}(x,0,0)}{g_z(x,0,0)}\tilde{g}(x,y,z) \\
\nonumber   &\quad\quad+ \frac{g_z(0,0,0)}{g_z(x,0,0)}[g_x(x,y,g^*(0,y,h-z))-g_x(x,0,0)],
  \end{align}
  and 
  \begin{align}
 \label{eq:newgy}   -\frac{\tilde{g}_y}{\tilde{g}_z}(x,y,z) = \frac{1}{g^*_u(0,y,h-z)}\frac{g_y}{g_z}(x,y,g^*(0,y,h-z)) + \frac{g^*_y}{g^*_u}(0,y,h-z).
  \end{align}
  Thus the A1 condition for $\overline{g}$ holds because, by \eqref{eq:newgx}, \textit{for fixed} $x$ the mapping $(y,z) \mapsto(\tilde{g}(x,y,z),\tilde{g}_x(x,y,z))$ is injective (by the A1 condition for $g$). Similar reasoning using \eqref{eq:newgy} yields the A1$^*$ condition for $\overline{g}$. More precisely by \eqref{eq:newgy} \textit{for fixed} $(y,z)$ the mapping $x\mapsto \frac{\tilde{g}_y}{\tilde{g}_z}(x,y,z)$ is injective by A1$^*$.

  Next, we introduce the notation $\overline{Y}(q,U,P),\overline{Z}(q,U,P)$ to denote $\overline{Y},\overline{Z}$ solving
  \begin{align*}
    \overline{g}(q,\overline{Y}(q,U,P),\overline{Z}(q,U,P)) = U,\\
    \overline{g}_q(q,\overline{Y}(q,U,P),\overline{Z}(q,U,P)) =P .
  \end{align*}
  We compute
  \begin{align}
 \nonumber \overline{Z}(q,U,P) = h - &g\left[0,\overline{Y}(x,U,P),g^*\left(x,\overline{Y}(x,U,P),\frac{g_z(x,0,0)}{g_z(0,0,0)}U+g(x,0,0)\right)\right]\\
   \label{eq:ny} \overline{Y}(q,U,P) &= p\Big[Y\Big(x,\frac{g_z(x,0,0)}{g_z(0,0,0)}U+g(x,0,0),\\
   \nonumber     &\quad\quad\frac{g_z(x,0,0)}{g_z(0,0,0)}\frac{\partial q}{\partial x}P+\frac{g_{x,z}(x,0,0)}{g_z(0,0,0)}U+g(x,0,0)\Big)\Big].
  \end{align}

  For the A2 condition first note the calculation $\tilde{g}_z<0$ implies $\overline{g}_z<0$. To check $\det \overline{E} \neq 0$ it suffices to check $ \det D_P\overline{Y} \neq 0$ \cite[eq. 2.2]{Trudinger14}. This follows from \eqref{eq:ny}. The key point is that, despite the unwieldy expression, we have $\overline{Y}(q,U,P) = p(Y(x,l_1(U),l_2(P)))$ for a nonsingular linear function $l_2(P)$. 

 For the Loeper maximum principle we must verify to verify for each $q,q',U$
  \begin{align*}
   \overline{g}(q',&\overline{Y}(q,U,P_\theta),\overline{Z}(q,U,P_\theta)) \\&\leq \max\{\overline{g}(q',\overline{Y}(q',U,P_0),\overline{Z}(q,U,P_0)),\overline{g}(q,\overline{Y}(q,U,P_1),\overline{Z}(q,U,P_1))\},
  \end{align*}
whenever $\{P_\theta\}_{\theta \in[0,1]} $ is a line segment for which the above quantities are well defined. This follows by a direct calculation using the definition of $\overline{g}$, \eqref{eq:ny}, and the Loeper maximum principle for $g$. 
\end{proof}

\begin{lemma}\label{lemma:lmp-gqq}
  Assume $g$ is a $C^3$ generating function satisfying the Loeper maximum principle. Then the statement of quantitative-quasiconvexity \eqref{eq:gqq} holds. 
\end{lemma}
\begin{proof}
 We'll use the notation preceding \eqref{eq:gqq}.  We note by recent work of Loeper and Trudinger \cite{LoeperTrudinger21} the Loeper maximum principle implies the well known A3w condition in the sense that the function
  \[ p \mapsto A_{ij}(x,u,p)\xi_i\xi_j,\]
  is convex on line segments orthogonal to $\xi$. Thus, working in the transformed $x$ coordinates with respect to $y_0,z_0$, we define the $g$-segment $x_\theta = \theta x_1+(1-\theta)x_0$ and put
  \[  h(\theta) := g(x_\theta,y_1,z_1) - g(x_\theta,y_0,z_0).\]
 Guillen and Kitagawa \cite[Lemma 9.3]{GuillenKitagawa17} prove  when $g$ is $C^4$ and satisfies A3w
  \begin{equation}
    \label{eq:h-ineq}
      h''(\theta) \geq -K\vert h'(\theta)\vert .
  \end{equation}
   We prove \eqref{eq:h-ineq} holds for $g\in C^3(\Gamma)$ satisfying A3w in the above sense by a minor modification of the proofs in \cite[2.11]{Trudinger14}, \cite[Appendix]{Rankin2020b} which are for $g \in C^4(\Gamma)$. In the proofs op. cit. $C^4$ differentiability is used to deal with a term
  \[ \big[A_{ij}(x_{\theta},u_0,p_1) -A_{ij}(x_{\theta},u_0,p_0)-D_{p_k}A_{ij}(x_{\theta},u_0,p_0)(p_1-p_0)\big](x_1-x_0)_i(x_1-x_0)_j.\]
  where $p_1 = g_x(x_\theta,y_1,z_1)$ and $p_0 = g_x(x_\theta,y_0,z_0)$. For merely $C^3$ $g$ we set $\overline{x}=x_1-x_0$, $\overline{p} = p_1-p_0$ and compute, by a Taylor series for $A_{ij}(x_{\theta},u_0,p_1-(1-t) (\overline{p}\cdot\overline{ x})\overline{x}/\vert\overline{x} \vert^2)$, the inequality
  \begin{align*}
    &\big[A_{ij}(x_{\theta},u_0,p_1) -A_{ij}(x_{\theta},u_0,p_0)-D_{p_k}A_{ij}(x_{\theta},u_0,p_0)\overline{p}\big]\overline{x}_i\overline{x}_j \\
    &\geq \big[A_{ij}(x_{\theta},u_0,p_1 - (\overline{p}\cdot \overline{x}) \overline{x}/\vert\overline{x} \vert^2) -A_{ij}(x_{\theta},u_0,p_0)\\
    &\quad\quad-D_{p_k}A_{ij}(x_{\theta},u_0,p_0)(\overline{p} - (\overline{p} \cdot x) \overline{x}/\vert\overline{x} \vert^2)\big]\overline{x}_i\overline{x}_j  - K \overline{p} \cdot \overline{x}.
  \end{align*}
  Here $K$ depends on $\vert D_pA_{ij} \vert$, i.e. $\Vert g \Vert_{C^{3}}$. Noting the line segment from $p_0$ to $p_1 - (\overline{p} \cdot \overline{x})\overline{x}/\vert\overline{x} \vert^2$ is orthogonal to $\overline{x}$ we obtain \eqref{eq:h-ineq} by the A3w condition and $\overline{p} \cdot \overline{x} = h'(\theta)$. Then \eqref{eq:gqq} is a consequence of \eqref{eq:h-ineq} \cite[Corollary 9.4]{GuillenKitagawa17}.
\end{proof}

For cost functions this is proved (under weaker assumptions) by Jeong \cite{jeong2020}.   The brevity of our proof is superficial: the real work is hidden in Loeper and Trudinger's result that LMP implies A3w for $C^2$ generating functions.

\bibliographystyle{plain}
\bibliography{thesis} %have mybib.bib in same directory

\end{document}